\documentclass[a4paper,twoside]{article}
\usepackage{a4}
\usepackage{amssymb}
\usepackage{amsmath}
\usepackage{upref}
\usepackage{bbm}
\usepackage{enumerate}
\usepackage[active]{srcltx}
\usepackage[dvips,colorlinks,citecolor=blue,linkcolor=blue]{hyperref}
\usepackage[dvipsnames]{color}
\newcommand{\ds}{\displaystyle}
\newcount\minutes \newcount\hours
\hours=\time
\divide\hours 60
\minutes=\hours
\multiply\minutes -60
\advance\minutes \time
\newcommand{\klockan}{\the\hours:{\ifnum\minutes<10 0\fi}\the\minutes}
\newcommand{\tid}{\today\ \klockan}
\newcommand{\prtid}{\smash{\raise 10mm \hbox{\LaTeX ed \tid}}}
\makeatletter
\def\sectionmark#1{} 
\def\subsectionmark#1{}
\newcommand{\sectnr}{\ifnum \c@secnumdepth >\z@
                 \thesection.\hskip 1em\relax \fi}
\def\@evenhead{\footnotesize\rm\thepage\hfil\leftmark\hfil\llap{\prtid}}
\def\@oddhead{\footnotesize\rm\rlap{\prtid}\hfil\rightmark\hfil\thepage}
\def\tableofcontents{\section*{Contents} 
 \@starttoc{toc}}
\makeatother
\makeatletter
\def\@biblabel#1{#1.}
\makeatother
\makeatletter
\let\Thebibliography=\thebibliography
\renewcommand{\thebibliography}[1]{\def\@mkboth##1##2{}\Thebibliography{#1}
\addcontentsline{toc}{section}{References}
\frenchspacing 
\setlength{\@topsep}{0pt}
\setlength{\itemsep}{0pt}%
\setlength{\parskip}{0pt plus 2pt}%
}
\makeatother
\makeatletter
\def\mdots@{\mathinner.\nonscript\!.%
 \ifx\next,.\else\ifx\next;.\else\ifx\next..\else
 \nonscript\!\mathinner.\fi\fi\fi}
\let\ldots\mdots@
\let\cdots\mdots@
\let\dotso\mdots@
\let\dotsb\mdots@
\let\dotsm\mdots@
\let\dotsc\mdots@
\def\vdots{\vbox{\baselineskip2.8\p@ \lineskiplimit\z@
    \kern6\p@\hbox{.}\hbox{.}\hbox{.}\kern3\p@}}
\def\ddots{\mathinner{\mkern1mu\raise8.6\p@\vbox{\kern7\p@\hbox{.}}%
    \raise5.8\p@\hbox{.}\raise3\p@\hbox{.}\mkern1mu}}
\makeatother
\makeatletter
\let\Enumerate=\enumerate
\renewcommand{\enumerate}{\Enumerate%
\setlength{\@topsep}{0pt}
\setlength{\itemsep}{0pt}%
\setlength{\parskip}{0pt plus 1pt}%
\renewcommand{\theenumi}{\textup{(\alph{enumi})}}%
\renewcommand{\labelenumi}{\theenumi}%
}
\let\endEnumerate=\endenumerate
\renewcommand{\endenumerate}{\endEnumerate\unskip}
\makeatother
\makeatletter
\def\@seccntformat#1{\csname the#1\endcsname.\quad}
\makeatother
\newcommand{\art}[6]{{\sc #1, \rm #2, \it #3 \bf #4 \rm (#5), \mbox{#6}.}}

\newcommand{\book}[3]{{\sc #1, \it #2, \rm #3.}}
\newcommand{\AND}{{\rm and }}
\RequirePackage{amsthm}
\newtheoremstyle{descriptive}%
  {\topsep}   
  {\topsep}   
  {\rmfamily} 
  {}          
  {\bfseries} 
  {.}         
  { }         
  {}          
\newtheoremstyle{propositional}%
  {\topsep}   
  {\topsep}   
  {\itshape}  
  {}          
  {\bfseries} 
  {.}         
  { }         
  {}          
\newtheoremstyle{remarkstyle}%
  {\topsep}   
  {\topsep}   
  {\rmfamily}  
  {}          
  {\itshape} 
  {.}         
  { }         
  {}          
\theoremstyle{propositional}
\newtheorem{thm}{Theorem}[section]
\newtheorem{prop}[thm]{Proposition}
\newtheorem{lem}[thm]{Lemma}
\newtheorem{cor}[thm]{Corollary}
\theoremstyle{descriptive}
\newtheorem{defn}[thm]{Definition}
\newtheorem{Assumption}[thm]{Assumption}
\newtheorem{egs}[thm]{Example}

\newtheorem{rem}[thm]{Remark}
\makeatletter
\renewenvironment{proof}[1][\proofname]{\par
  \pushQED{\qed}%
  \normalfont 
  \trivlist
  \item[\hskip\labelsep
        \itshape
    #1\@addpunct{.}]\ignorespaces
}{%
  \popQED\endtrivlist\@endpefalse
}
\makeatother
\newdimen\extrawidth
\def\iintlim#1#2{\setbox0\hbox{$\scriptstyle#1$}%
        \setbox1\hbox{$\scriptstyle#2$}%
        \extrawidth=\wd1 \advance\extrawidth-\wd0
        \ifdim\extrawidth<0pt \extrawidth=0pt\fi%
        \int_{#1\kern\extrawidth \kern .5em}^{#2\kern -\wd1} \kern -.5em%
}
\newcommand{\cprime}{{\mathsurround0pt$'$}}
\renewcommand{\emptyset}{\varnothing}
\def\vint{\mathop{\mathchoice%
          {\setbox0\hbox{$\displaystyle\intop$}\kern 0.22\wd0%
           \vcenter{\hrule width 0.6\wd0}\kern -0.82\wd0}%
          {\setbox0\hbox{$\textstyle\intop$}\kern 0.2\wd0%
           \vcenter{\hrule width 0.6\wd0}\kern -0.8\wd0}%
          {\setbox0\hbox{$\scriptstyle\intop$}\kern 0.2\wd0%
           \vcenter{\hrule width 0.6\wd0}\kern -0.8\wd0}%
          {\setbox0\hbox{$\scriptscriptstyle\intop$}\kern 0.2\wd0%
           \vcenter{\hrule width 0.6\wd0}\kern -0.8\wd0}}%
          \mathopen{}\int}
\def\vintslides{\mathop{\mathchoice%
          {\setbox0\hbox{$\displaystyle\intop$}\kern 0.22\wd0%
           \vcenter{\hrule height 0.04em width 0.6\wd0}\kern -0.82\wd0}%
          {\setbox0\hbox{$\textstyle\intop$}\kern 0.2\wd0%
           \vcenter{\hrule height 0.04em width 0.6\wd0}\kern -0.8\wd0}%
          {\setbox0\hbox{$\scriptstyle\intop$}\kern 0.2\wd0%
           \vcenter{\hrule height 0.04em width 0.6\wd0}\kern -0.8\wd0}%
          {\setbox0\hbox{$\scriptscriptstyle\intop$}\kern 0.2\wd0%
           \vcenter{\hrule height 0.04em width 0.6\wd0}\kern -0.8\wd0}}%
          \mathopen{}\int}
\DeclareMathOperator{\dist}{dist}

\newcommand{\dinn}{d_{\text{inn}}}

\DeclareMathOperator{\spt}{supp}
\newcommand{\supp}{\spt}

\DeclareMathOperator*{\essliminf}{ess\,lim\,inf}

\newcommand{\bdry}{\partial}
\newcommand{\bdy}{\bdry}

{\catcode`p =12 \catcode`t =12 \gdef\eeaa#1pt{#1}}      
\def\accentadjtext#1{\setbox0\hbox{$#1$}\kern   
                \expandafter\eeaa\the\fontdimen1\textfont1 \ht0 }
\def\accentadjscript#1{\setbox0\hbox{$#1$}\kern 
                \expandafter\eeaa\the\fontdimen1\scriptfont1 \ht0 }
\def\accentadjscriptscript#1{\setbox0\hbox{$#1$}\kern   
                \expandafter\eeaa\the\fontdimen1\scriptscriptfont1 \ht0 }
\def\accentadjtextback#1{\setbox0\hbox{$#1$}\kern       
                -\expandafter\eeaa\the\fontdimen1\textfont1 \ht0 }
\def\accentadjscriptback#1{\setbox0\hbox{$#1$}\kern     
                -\expandafter\eeaa\the\fontdimen1\scriptfont1 \ht0 }
\def\accentadjscriptscriptback#1{\setbox0\hbox{$#1$}\kern 
                -\expandafter\eeaa\the\fontdimen1\scriptscriptfont1 \ht0 }
\def\itoverline#1{{\mathsurround0pt\mathchoice
        {\rlap{$\accentadjtext{\displaystyle #1}
                \accentadjtext{\vrule height1.593pt}
                \overline{\phantom{\displaystyle #1}
                \accentadjtextback{\displaystyle #1}}$}{#1}}
        {\rlap{$\accentadjtext{\textstyle #1}
                \accentadjtext{\vrule height1.593pt}
                \overline{\phantom{\textstyle #1}
                \accentadjtextback{\textstyle #1}}$}{#1}}
        {\rlap{$\accentadjscript{\scriptstyle #1}
                \accentadjscript{\vrule height1.593pt}
                \overline{\phantom{\scriptstyle #1}
                \accentadjscriptback{\scriptstyle #1}}$}{#1}}
        {\rlap{$\accentadjscriptscript{\scriptscriptstyle #1}
                \accentadjscriptscript{\vrule height1.593pt}
                \overline{\phantom{\scriptscriptstyle #1}
                \accentadjscriptscriptback{\scriptscriptstyle #1}}$}{#1}}}}
\def\itunderline#1{{\mathsurround0pt\mathchoice
        {\rlap{$\underline{\phantom{\displaystyle #1}
                \accentadjtextback{\displaystyle #1}}$}{#1}}
        {\rlap{$\underline{\phantom{\textstyle #1}
                \accentadjtextback{\textstyle #1}}$}{#1}}
        {\rlap{$\underline{\phantom{\scriptstyle #1}
                \accentadjscriptback{\scriptstyle #1}}$}{#1}}
        {\rlap{$\underline{\phantom{\scriptscriptstyle #1}
                \accentadjscriptscriptback{\scriptscriptstyle #1}}$}{#1}}}}
\newcommand{\be}{\beta}

\newcommand{\eps}{\varepsilon} 
\newcommand{\ga}{\gamma}

\newcommand{\Om}{\Omega}
\renewcommand{\phi}{\varphi}
\newcommand{\p}{{$p\mspace{1mu}$}}   
\newcommand{\R}{\mathbb{R}}

\newcommand{\N}{\mathbb{N}}
\newcommand{\eR}{{\overline{\R}}}

\newcommand{\sm}{\setminus}

\newcommand{\limminus}{{\mathchoice{\raise.17ex\hbox{$\scriptstyle -$}}
                {\raise.17ex\hbox{$\scriptstyle -$}}
                {\raise.1ex\hbox{$\scriptscriptstyle -$}}
                {\scriptscriptstyle -}}}
\newcommand{\limplus}{{\mathchoice{\raise.17ex\hbox{$\scriptstyle +$}}
                {\raise.17ex\hbox{$\scriptstyle +$}}
                {\raise.1ex\hbox{$\scriptscriptstyle +$}}
                {\scriptscriptstyle +}}}
\newcommand{\Np}{N^{1,p}}

\newcommand{\OmPC}{\overline{\Om}^P}

\newcommand{\peto}{\stackrel{\OmPC}{\rightarrow}}

\newcommand{\CpP}{\itoverline{C}_p^P}

\makeatletter
\newcommand{\setcurrentlabel}[1]{\def\@currentlabel{#1}}
\makeatother

\numberwithin{equation}{section}

\begin{document}

\title{Geometry of prime end boundary and the Dirichlet problem for bounded domains in metric measure spaces}
\author{
Dewey Estep and Nageswari Shanmugalingam \footnote{N.S.~was partially supported by NSF grant~\#DMS-1200915.
Part of the work was conducted during the stay of the authors at the Institute of Pure and Applied 
Mathematics (IPAM); the authors thank that institution
for its kind hospitality. The authors also wish to thank the anonymous referee for helpful suggestions that improved the 
exposition of the paper.}
}

\maketitle

\noindent{\small
{\bf Abstract}. In this note we study the Dirichlet problem associated with a version of prime end boundary of a 
bounded domain in a complete metric measure space equipped with a doubling measure supporting a Poincar\'e 
inequality. We show the resolutivity of functions that are continuous on the prime end boundary and are 
Lipschitz regular when restricted to the subset of all prime ends whose impressions are singleton sets. We also
consider a new notion of capacity adapted to the prime end boundary, and show that bounded perturbations of
such functions on subsets of the prime end boundary with zero capacity are resolutive and that their Perron solutions
coincide with the Perron solution of the original functions. We also describe some examples which demonstrate
the efficacy of the prime end boundary approach in obtaining new results even for the classical Dirichlet problem for
some Euclidean domains.} 

\bigskip
\noindent
{\small \emph{Key words and phrases}: 
Prime end boundary, Dirichlet problem, $p$-harmonic functions, Perron method, metric measure spaces, doubling measure,
Poincar\'e inequality.}

\medskip
\noindent
{\small Mathematics Subject Classification (2010): 
Primary: 31E05; Secondary: 31B15, 31B25, 31C15, 30L99.
}

\section{Introduction} 

 The Dirichlet problem associated with a partial differential operator $L$ on a domain $\Om$ is the
 problem of finding a function $u$ on $\Om$ such that $Lu=0$ (usually in a weak sense) on $\Om$ 
 and $u-f\in W^{1,p}_0(\Om)$ for a given boundary data $f:\bdy\Om\to\R$. However, in some situations
 $\bdy\Om$ is not the correct boundary to be considered. For example, given a flat metal disc,
 if we cut a radial slit in the disc and insert a non-conducting material in the slit, then heat energy cannot
 pass from one side of the slit to the other directly, and so in this case the correct boundary for the 
 slit disc (when the operator $L$ is the one associated with the heat equation) should count each
 point on the slit twice, once for each side of the slit. For more complicated domains the corresponding 
 natural boundary is more complicated. To address this issue, the paper~\cite{ABBS} proposed an
 alternative for the topological boundary $\bdy\Om$, called the prime end boundary. The goal of 
 this note is to use the prime end boundary in the study of the Dirichlet problem. 
 
 In this note we consider a variational analog of the $p$-Laplacian $\Delta_p$ in the setting of bounded domains
 in complete metric measure spaces equipped with a doubling measure supporting a \p-Poincar\'e inequality.
 We use the Perron method to construct solutions to the Dirichlet problem on bounded domains in such metric
 spaces. 
 
 The Perron method was successfully used in~\cite{BBSPer} to construct solutions to the 
 Dirichlet problem in the metric setting when the boundary considered is the topological boundary. We demonstrate
 in this paper that such a method also works for the prime end boundary. The paper~\cite{BBS}
 considered the Dirichlet problem for the prime end boundary in the simple situation that each prime end
 has only one point in its impression (see Section~2 for the definitions of these concepts) and that the 
 prime end boundary is compact. However, in general the prime end boundary is not compact, as even the 
 simple example of the harmonic comb shows (see the examples in Section~8). Hence the 
 principal part of the work of this note is to overcome the non-compactness issue of the prime end boundary
 in applying the Perron method. 
 
 The standard assumption in this paper is that the metric space under study is a complete metric space 
 equipped with a doubling measure supporting a \p-Poincar\'e inequality for some fixed $1<p<\infty$. We 
 use the Newtonian spaces as substitutes for Sobolev spaces under this assumption quite successfully, but
 we point out that an alternate construction of Sobolev-type spaces has been used by others
 successfully in some fractal spaces that do not satisfy the Poincar\'e inequality considered here; see
 for example~\cite{DSV} and the references therein. The paper~\cite{LP} considers the Dirichlet problem
 associated with the $p$-Laplacian on domains in metric measure spaces that satisfy the doubling and 
 Poincar\'e inequality assumptions; however, the boundary that paper studies is the $p$-Royden algebra boundary,
 which is a functional analytic construct. The notion of prime end boundary we consider in this paper is more
 of a geometric construct. 
 
 The structure of this paper is as follows. In Section~2 we explain the notation used in this paper, in particular
 we explain the construction of the prime end boundary. One should keep in mind that even in the setting of 
 simply connected planar domains the prime end boundary described in Section~2 could differ from that of
 Carath\'eodory (see~\cite{ABBS}), but it has the advantage of being usable for non-simply connected planar domains
 and more general domains in higher dimensional Euclidean spaces as well as domains in metric measure spaces.
 In Section~3 we continue the explanation of concepts used by describing the analog of Sobolev spaces in the
 metric setting, called the Newtonian spaces, and by describing the relevant associated potential theory. 
 In Section~4 we explore some structures associated with the prime end boundaries, and in Assumption~\ref{St-Assume}
 we give a natural condition on the domain needed in the rest of the paper. Many domains whose prime end
 boundaries are not compact do satisfy this condition (see the examples in Section~8 for a sampling), but
 we do not know of any domain that would violate this condition. 
 
 In Section~5 we gather some additional
 properties of prime end boundaries of domains that satisfy  Assumption~\ref{St-Assume}, including
 the key property that if the boundary of 
 a connected open subset of $\Om$ intersects the topological boundary of $\Om$, then it must,
 under the prime end closure topology, intersect the prime end boundary of $\Om$; see
 Theorem~\ref{main1}. In Section~6 we propose a
 modification of the $p$-capacity used in~\cite{BBS}, adapted to the prime end boundary, and study its basic
 capacitary properties. We also show in this section that functions in the Newtonian class of the domain 
 with zero boundary values (denoted
 $N^{1,p}_0(\Om)$) are quasicontinuous with respect to this new capacity. In Section~7 we use the above 
 notions together with the Perron method, adapted to the prime end boundary, to obtain resolutivity properties
 of certain continuous functions on the prime end boundary of $\Om$. We also show stability of the Perron
 solution under bounded perturbation of these functions on sets of (new) capacity zero. Finally, in
 Section~8 we describe three examples and use them to show how, even in the Euclidean setting, new
 stability results for the classical Dirichlet problem can be obtained from the prime end boundary approach.

\section{Preliminaries: the prime end boundary}
\label{sect-prelim}

In this paper we assume that $(X,d)$ is a complete, doubling metric space that is quasiconvex. 
Recall that $X$ is quasiconvex if there is a constant $C_q\ge 1$ such that whenever $x,y\in X$, there is
a rectifiable curve (that is, a curve of finite length) $\gamma$ with end points $x$ and $y$ such that
the length of $\gamma$, denoted $\ell(\gamma)$, is at most $C_q\, d(x,y)$.
Quasiconvexity is a consequence of the validity of a $p$-Poincar\'e inequality on the metric measure space
$(X,d,\mu)$ when $\mu$ is doubling, and from Section~6 onward we will assume that $\mu$ is doubling
and supports a $p$-Poincar\'e inequality. So the assumption of quasiconvexity here is not 
overly constrictive. Furthermore, complete doubling metric spaces have a highly useful topological property
called properness. A metric space is proper if closed and bounded subsets of the space are compact. This property
will enable us to apply the Arzela-Ascoli theorem in subsequent sections of this paper. To see that 
a complete metric space $X$ equipped with a doubling measure is proper, we may argue as follows. 
Since $X$ is equipped with a doubling measure, it is doubling in the sense of~\cite[Section~10.13]{Hei}, 
see~\cite[page~82]{Hei}. It follows that closed balls in such a space $X$ are complete and totally
bounded, and so are compact; see~\cite[page~275, Theorem~3.1]{Munk}.

We essentially follow~\cite{ABBS} in the construction of prime ends for bounded domains in $X$. In what
follows, $\Om\subset X$ is a bounded open connected set.

In addition to the standard metric balls $B(x,r):=\{y\in X\, :\,  d(x,y)<r\}$, we will also make use of the 
\emph{$r$-neighborhood} of a set, defined as 
\begin{equation}\label{e2.1}
    N(A,r):=\bigcup_{x\in A}B(x,r).
\end{equation}

We will also use the notion of the distance from a point to a set and distance between two
sets: 
\[
\text{dist}(x,A):=\inf\{d(x,y)\, :\, y\in A\}, 
   \ \ \ \text{dist}(A,B):=\inf\{d(x,y)\, :\, x\in A, y\in B\}.
\]

Since $X$ is quasiconvex, 
it is easy to see
by a topological argument that
an open connected  subset of $X$ is rectifiably connected. 
A proof of this appears in~\cite[Lemma~4.38]{BBbook}, but since the proof is elementary, we also give a proof here.
Indeed, given an open connected subset $U$ of $X$ and $x\in U$, consider the collection $U(x)$ 
of all points $y\in U$ such that 
there is a rectifiable curve in $U$ connecting $x$ to $y$. Quasiconvexity of $X$ implies that whenever $y\in U(x)$,
there is a ball centered at $y$ contained in $U(x)$. Thus $U(x)$ is an open subset of $U$. Similar argument gives
$U\setminus U(x)$ is also open, and since $U$ is connected, this means that either $U(x)$ is empty or $U(x)=U$.
Because $x\in U(x)$, it follows that $U(x)=U$, and so $U$ is rectifiably connected.  

\begin{defn}\label{defn:dinn}
 Given a set $U\subset X$, the \emph{inner distance} on $U$ is given  for $x,y\in U$ by 
\[
  \dinn^U(x,y)=\inf_\ga \ell(\ga),
\]
where the infimum is taken over rectifiable curves $\ga$ in $U$ with end points $x,y$.
\end{defn}
If $U$ is not connected and $x,y$ belong to different components of $U$, then we have $\dinn^U(x,y)=\infty$.
However, if $U$ is a connected open subset of $X$, then, by the comments before the above definition, we 
know that $\dinn^U$ is a metric on $U$. Given that $X$ is complete and proper, an application of the 
Arzel\`{a}-Ascoli theorem tells us that if $\dinn^U(x,y)$ is finite, then there is a $\dinn^U$-geodesic $\ga_{x,y}^U$
connecting $x$ to $y$ in $\overline{U}$ with length $\ell(\ga_{x,y}^U)=\dinn^U(x,y)$. 
Here, by a $\dinn^U$-geodesic we mean a curve in $\overline{U}$ connecting $x$ to $y$ that appears
as a uniform limit of a sequence of length-minimzing curves in $U$ connecting $x$ to $y$.
Furthermore, the 
quasiconvexity of $X$ implies that, if $U$ is open, then, for each $x\in U$ with $r=\dist(x,X\setminus U)/C_q$,
the two metrics $d$ and $\dinn^U$ are biLipschitz equivalent on $B(x,r)$ with biLipschitz constant $C_q$.

We will 
make use of the Mazurkiewicz distance, defined below.

\begin{defn}\label{dm}
Let $\Om$ be a bounded connected open subset of $X$, that is, $\Om$ is a bounded domain.
Given $x,y\in\Om$, the \emph{Mazurkiewicz distance $d_M$ between $x$ and $y$ on $\Om$} is
\[
d_M(x,y)=\inf_E\, \text{diam }E,
\]
where the infimum is taken over all connected sets $E\subset\Om$ with $x,y\in E$.
\end{defn}

It is clear that $d_M$ is a metric on $\Om$, with $d(x,y)\leq d_M(x,y)\leq \dinn^\Om(x,y)$. 
The completion of $\Om$ under $d_M$ 
is denoted $\overline{\Om}^M$, with $\bdy_M\Om:=\overline{\Om}^M\sm\Om$. The metric $d_M$ 
extends naturally to a metric on $\overline{\Om}^M$; this extended metric will also be denoted by $d_M$.

Note that~\eqref{e2.1} can be applied to the distances in Definitions~\ref{defn:dinn} and~\ref{dm},
with the new $r$-neighborhoods being denoted $N_{inn}^U(x,r)$ and $N_M(x,r)$ respectively.

\begin{defn}\label{def:end}
A set $E\subset\Om$ is \emph{acceptable} if $E$ is connected and $\overline{E}\cap\bdy\Om$ is
non-empty. A sequence $\{E_k\}_{k\in\N}$ of acceptable sets is a \emph{chain} if all of the following 
conditions hold true:
\begin{enumerate}
\item $E_{k+1}\subset E_k$ for $k\in\N$,
\item for each $k\in\N$, the distance $\dist_M(\Om\cap\bdy E_k, \Om\cap\bdy E_{k+1})>0$,
\item the \emph{impression} $I(\{E_k\}_k):=\bigcap_{k\in\N}\overline{E_k}$ is a subset of $\bdy\Om$.
\end{enumerate} 
\end{defn}

Note that $I(\{E_k\}_k)$ is a compact, connected set.

Our definition differs slightly from that given in~\cite{ABBS} in that condition (b) now references the Mazurkiewicz
distance. However, the examples and results of~\cite{ABBS} still hold. Indeed, whenever 
the analog of condition~(b) was used in~\cite{ABBS} to prove a claim, the key property used was that when
$\{E_k\}_{k\in\N}$ is a chain, for each $k$ and points $x\in E_{k+1}$ and $y\in \Omega\setminus E_k$, every 
connected compact subset of $\Omega$ that contains both $x$ and $y$ must have diameter bounded below by
a positive number that may depend on $k$ but not on $x$, $y$. This is precisely the condition given by 
\emph{our} version of condition~(b), and so the results of~\cite{ABBS} hold for our ends as well. 
The principal result of~\cite{ABBS} we depend on is the identification of ends that have singleton impressions as
certain prime ends. For the convenience of the reader we will give a proof of that fact here; see  Lemma~\ref{lem:single}
below. The examples given
in~\cite{ABBS} are simple enough that it can be directly verified that the prime ends for those example domains in
the sense of~\cite{ABBS} are the same as those in our sense. While we have more chains than~\cite{ABBS}
the additional chains are equivalent (in the sense described in Definition~\ref{def:divide} below) 
to chains that satisfy the conditions of~\cite{ABBS}. 

It is easy to see that a chain in the sense of~\cite{ABBS} is a chain in our sense, but the converse need not be true.
Therefore in general we have more chains in the sense of Definition~\ref{def:end} than does~\cite{ABBS}. Therefore
conceivably we have more ends than does~\cite{ABBS} and thus an end that might be prime in the setting of~\cite{ABBS}
(see the definition of prime ends below) may not be prime in our sense. However, given that the notion of Sobolev spaces
in the metric setting considered here uses paths extensively, the Mazurkiewicz distance seems to be the natural
one to consider here. We point out that we are in good company here; it was shown by N\"akki that condition~(b)
is equivalent to an Ahlfors-type condition regarding extremal length when the domain is a quasiconformally collared
Euclidean domain, see~\cite{Nak}.

We have also chosen to use the Mazurkiewicz distance $\dist_M$ rather than the original metric distance $d$ (as 
\cite{ABBS} does) because in constructing ends that intersect certain open subsets of $\Omega$, it is 
easier to describe the construction when $\dist_M$ is used rather than $\dist$; see Section~5. Thus the use
of $\dist_M$ makes for a simpler exposition, and hence we have chosen to give the above modification. We again
point out that whenever the analog of condition~(b) was used in a proof in~\cite{ABBS}, it is actually the
positivity of $\dist_M$-distance that was needed. Hence we do not lose anything by our modification.

\begin{defn}\label{def:divide}
Given two chains $\{E_k\}_k$ and $\{F_k\}_k$, we say that $\{E_k\}_k$ \emph{divides} $\{F_k\}_k$ if,
for each positive integer $k$, there is a positive integer $j_k$ such that $E_{j_k}\subset F_k$. 
\end{defn}

The above notion of division gives an equivalence relationship on the collection of all chains; two chains
$\{E_k\}_k$ and $\{F_k\}_k$ are equivalent if they both divide each other. Given a
chain $\{E_k\}_k$, its equivalence class is denoted $[\{E_k\}_k]$. If two chains $\{E_k\}_k$ and $\{F_k\}_k$
are equivalent, then their impressions are equal. Let this (common) impression be denoted
$I[\{E_k\}_k]$. These equivalence classes are called \emph{ends} of $\Om$. The collection
of all ends of $\Om$ is called the \emph{end boundary} $\bdy_E\Om$ of $\Om$.

Observe also that if a chain $\{E_k\}_k$ divides another chain $\{G_k\}_k$, and $\{F_k\}_k\in[\{E_k\}_k]$, then
$\{F_k\}_k$ also divides $\{G_k\}_k$. Furthermore, $\{E_k\}_k$ divides every chain in $[\{G_k\}_k]$. Hence the
notion of divisibility extends to ends as well. We take as notation $[\{E_k\}_k]\Big|[\{G_k\}_k]$ to mean that $[\{E_k\}_k]$ 
divides $[\{G_k\}_k]$.

\begin{defn}
An end of $\Om$ is a \emph{prime end} if the only end that divides it is itself. The collection of all 
prime ends of $\Om$, called the \emph{prime end boundary of} $\Om$, is denoted $\bdy_P\Om$. The collection of all 
prime ends of $\Om$
with singleton impression is called the \emph{singleton prime end boundary} and is denoted $\bdy_{SP}\Om$.
\end{defn}

\begin{lem}\label{lem:single}
Let $\{E_k\}_k$ be a chain such that $I(\{E_k\}_k)=\{x_0\}$, that is, the chain has only a singleton impression. Then
$[\{E_k\}_k]$ is a prime end, and for each positive integer $k$ there is a positive real number $r_k>0$ such that a connected
component of $B(x,r_k)\cap\Omega\subset E_k$. 
\end{lem}

\begin{proof}
We will prove the second part of the lemma, for then the first part follows from~\cite[Lemma~7.3]{ABBS} and
\cite[Corollary~7.11]{ABBS} (see also the discussion in~\cite[Section~10]{ABBS} and~\cite{BBS3}). 

Suppose that there is no such positive number $r_k$. Then for each $r>0$ let $F_k(r)$ be the connected component
of $B(x_0,r)\cap\Omega$ containing points $x_k^r\in E_{k+1}$. Since $F_k(r)\not\subset E_k$, it follows that 
there is a point $y_k^r\in F_k(r)\setminus E_k$. Because $F_k(r)$ is a connected open subset of the quasi convex
space $X$, it follows that $F_k(r)$ is rectifiably connected (see the discussion before Definition~\ref{defn:dinn}).
Thus there is a compact curve $\gamma$ in $F_k(r)$ connecting $x_k^r\in E_{k+1}$ to $y_k^r\not\in E_k$, and the 
diameter of such a curve is at most $2r$. Thus 
\[
  0<\dist_M(\Om\cap\bdy E_k, \Om\cap\bdy E_{k+1})\le \text{diam}(\gamma)\le 2r,
\] 
the above inequality holding for each $r>0$. This is not possible. Hence such a positive number $r_k$ must exist.
\end{proof}

The following series of definitions describes a topology on $\bdy_E\Om$ that meshes well with the topology of $\Om$.
We first ``stitch" $\partial_E\Omega$ to $\Omega$ via a sequential topology as follows.

\begin{defn}
Given a sequence $\{x_i\}_i$ in $\Om$, we say that 
$x_i\peto[\{E_k\}_k]$ if for every positive integer
$k$ there is a positive integer $i_k$ such that whenever $i\ge i_k$ we have $x_i\in E_k$.
\end{defn}

One should be aware that a sequence in $\Om$ can converge to two \emph{different} ends, as 
\cite[Example~8.9]{ABBS} shows.

We next extend the topology to $\partial_E\Om$ by describing sequential toplogy on $\partial_E\Om$.

\begin{defn}
Given a sequence
$\{[\{E_k^n\}_k]\}_n$ of ends of $\Om$ and an end $[\{E_k^\infty\}_k]$ of $\Om$, we say that
$[\{E_k^n\}_k]\peto[\{E_k^\infty\}_k]$ if for each positive integer $k$ there is a positive integer $n_k$ such that
whenever $n\ge n_k$, there is a positive integer $j_n$ such that $E_{j_n}^n\subset E_k^\infty$.
\end{defn}

A modification of~\cite[Example~8.9]{ABBS} shows that a sequence of ends can converge to more than one
end. However, a sequence of ends will never converge to a point in $\Om$.

\begin{defn}
Equip the set $\overline{\Om}^E:=\Om\cup\bdy_E\Om$ with the sequential topology associated with the above notion of limits.
Equip the subset $\overline{\Om}^P:=\Om\cup\bdy_P\Om$ with the subspace topology inherited from $\overline{\Om}^E$. We call the sets $\overline{\Om}^E$
and $\overline{\Om}^P$ the {\it End Closure of $\Om$} and the {\it Prime End Closure of $\Om$} respectively.
\end{defn}

Sometimes, it may be useful to talk about the closure or boundary of a set $V\subset\overline{\Om}^P$ with respect to the
Prime End topology of $\Om$. To avoid confusion we will denote the {\it Prime End closure of $V$ with respect to the Prime End topology
on $\Om$} as $\overline{V}^{P,\Om}$ and the {\it Prime End boundary of $V$ with respect to the Prime End topology of $\Om$} as $\bdy_P^\Om V$.
Note that if $\overline{V}\subset \Om$, then $\overline{V}^{P,\Om}=\overline{V}$ and $\bdy_P^\Om V=\bdy V$.

%

\begin{rem}
Recall that by $\bdy_{SP}\Om$ we mean the collection of all prime ends of $\Om$ whose impressions contain only one 
point.
Recall the Mazurkiewicz boundary $\bdy_M\Om$ of $\Om$ from Definition~\ref{dm}.
Though $\OmPC$ admits no metric, it is shown in \cite[Theorem~9.5]{ABBS} that there is a homeomorphism 
$\Phi:\Om\cup\bdy_{SP}\Om\to\overline{\Om}^M$ such that $\Phi|_\Om$ is the identity map and 
$\Phi|_{\bdy_{SP}\Om}:\bdy_{SP}\Om\to\bdy_M\Om$. It follows that $\Om\cup\bdy_{SP}\Om$ is metrizable via the 
pullback of the metric $d_M$.
So, for $x,y\in\Om\cup\bdy_{SP}\Om$, by $d_M(x,y)$ we truly mean $d_M(\Phi(x),\Phi(y))$.
\end{rem}

\begin{rem}\label{rem:pebasis}
Given a set $G\subset\Om$, we define
\[
G^P:=G\cup\{[\{E_k\}_k]\in\bdry_P\Om\ |\text{ for some }j, \, E_j\subset G\}.
\]
It was shown in~\cite[Proposition~8.5]{ABBS} that the collection of sets
\[
\{G,G^P\ |\ G\subset\Om\text{ is open}\}
\]
forms a basis for the topology on $\overline{\Om}^P$.  Note that given the above definition of $G^P$, we have
$\overline{\Om}^P=\Om^P$.
In the next few sections, we will focus on the sequential definition of this topology. In later sections, the above natural
basis will prove invaluable in making our results more intuitive.
\end{rem}

\begin{defn}\label{accthrough}
We say that a point $x_0\in\bdy\Om$ is \emph{accessible from $\Om$} if there is a curve 
$\ga:[0,1]\to\overline{\Om}$ such that $\ga(1)=x_0$ and $\ga([0,1))\subset\Om$. We say that a point $x_0\in\bdy\Om$ is
\emph{accessible through the chain $\{E_k\}_k$} if there is such a curve $\ga$  satisfying in addition that 
for each positive integer $k$
there is some $0<t_k<1$ with $\ga([t_k,1))\subset E_k$. The curve $\ga$ is said to \emph{access $x_0$ through $\{E_k\}$}.
\end{defn} 

It is easy to see that if $x_0$ is accessible through $\{E_k\}_k$ and $\{F_k\}_k\in[\{E_k\}_k]$, then $x_0$ is 
accessible through $\{F_k\}_k$ as well. Furthermore, $x_0\in I[\{E_k\}_k]$. Thus, we can extend the above definitions to ends.
It was shown in~\cite{ABBS} that if $z_0\in\bdy\Om$ is accessible, then it is accessible through some 
prime end $[\{E_k\}_k]$ with $I[\{E_k\}_k]=\{x_0\}$. In addition, for all prime ends $[\{E_k\}_k]\in\bdy_{SP}\Om$, the point in
$I[\{E_k\}_k]$ is accessible through $[\{E_k\}_k]$.

However, as examples in~\cite{ABBS} show, for some domains $\Omega$, 
not all points in $\bdy\Om$ are accessible from $\Om$, and it is not
 true that $\bdy_P\Om$ is always compact. This has implications to the application of the Perron method in solving
Dirichlet problems for the boundary $\bdy_P\Om$, and the goal of this paper is to find a way to overcome this lack of compactness; the key lemma in this direction is Lemma~\ref{lem:Ass(a)}.

\begin{defn}\label{accside}
Let $V\subset\Om$ be an open connected set. We say that a point $x_0\in\bdy\Om$ is 
\emph{accessible from the side of} $V$ if there is a curve $\ga:[0,1]\to\overline{\Om}$ such that
$\ga([0,1))\subset\Om$, $\ga(1)=x_0$, and for each positive integer $n$ there is a real number
$t_n$ with $1-\tfrac1n<t_n<1$ such that $\ga(t_n)\in V$. We say that a chain $\{E_k\}_k$ of $\Om$ is 
\emph{from the side of} $V$ if $E_k\cap V$ is non-empty for each positive integer $k$. 
\end{defn}

Note that if $\{E_k\}_k$ is from the side of $V$, and $\{F_k\}_k\in[\{E_k\}_k]$, then $\{F_k\}_k$ is also 
from the side of $V$. Hence the property of being \emph{from the side of} $V$ is inherited
from chains by ends.

{\begin{rem}\label{curve-param}
In this paper, when we discuss curves $\ga$ that are locally rectifiable, we assume that $\ga$ is essentially
arc-length parametrized; that is, $\ga:[0,\infty)\to X$ such that $\ga\vert_{[0,\ell(\ga))}$ is arc-length parametrized,
and if $\ell(\ga)<\infty$, then for $t\ge \ell(\ga)$ we have $\ga(t)=\ga(\ell(\ga))$. We call such parametrizations 
\emph{standard parametrizations}.
\end{rem}

Note that in Definitions \ref{accthrough} and \ref{accside}, we could take $\ga$ to be maps from $[0,\infty)$ rather than
from $[0,1]$. In this case, in Definition~\ref{accside} we require
 $\ell(\gamma)-1/n< t_n<\ell(\gamma)$ whenever $\ell(\gamma)<\infty$, and 
 $n<t_n<\ell(\gamma)$ when $\ell(\gamma)=\infty$, rather than
$1-1/n<t_n<1$.

\section{Preliminaries: Newton-Sobolev spaces and potential theory}

We follow~\cite{Sh-rev} in considering the Newtonian spaces as the analog of Sobolev spaces in the metric setting.
Given a function $u:X\to[-\infty,\infty]$, we say that a non-negative Borel measurable function $g$ on $X$ is an
\emph{upper gradient} of $u$ if whenever $\gamma$ is a non-constant compact rectifiable curve in $X$, we have
\[
  |u(x)-u(y)|\le \int_\gamma g\, ds,
\]
where $x$ and $y$ denote the two end points of $\gamma$. The above inequality should be interpreted to mean that
$\int_\gamma g\, ds=\infty$ if at least one of $u(x)$, $u(y)$ is not finite. The notion of upper gradients is originally 
due to Heinonen and Koskela \cite{HK}, where it was called a very weak gradient. Of course, if $g$ is an upper 
gradient of $u$ and $\rho$ is a non-negative Borel measurable function on $X$, then $g+\rho$ is also an upper 
gradient of $u$. If $u$ has an upper gradient that belongs to $L^p(X)$, then the collection of all upper gradients of
$u$ in $L^p(X)$ forms a convex subset of $L^p(X)$. Therefore, by the uniform convexity of $L^p(X)$ when $1<p<\infty$ 
there is a unique function $g_u\in L^p(X)$ that is in the $L^p$-closure of this convex set, with minimal norm. 
Such a function $g_u$ is called the \emph{minimal $p$-weak upper gradient} of $u$.

Given $1<p<\infty$, the Newtonian space $N^{1,p}(X)$ is the space {\small
\[
 N^{1,p}(X):=\lbrace u:X\to [-\infty,\infty]\, :\, \int_X|u|^p\, d\mu<\infty,\text{ has an upper gradient }g\in L^p(X)
 \rbrace / \sim,
\]}
where the equivalence relationship $\sim$ is such that $u\sim v$ if and only if 
\[
  \Vert u-v\Vert_{N^{1,p}(X)}:=\left[\int_X|u-v|^p\, d\mu+\inf_g\int_Xg^p\, d\mu\right]^{1/p}=0,
\]
the infimum being taken over all upper gradients $g$ of $u-v$. See~\cite{Sh-rev} or~\cite{BBbook} for a discussion on 
the properties of $N^{1,p}(X)$. Just as sets of measure zero are exceptional sets in the $L^p$-theory, sets of
$p$-capacity zero are exceptional sets in the potential theory associated with $N^{1,p}(X)$. Given a set
$A\subset X$, its \emph{$p$-capacity} is the number
\[
   C_p(A:X):=\inf _u \Vert u\Vert_{N^{1,p}(X)}^p,
\]
where the infimum is taken over all $u\in N^{1,p}(X)$ that satisfy $u\ge 1$ on $A$.

\begin{defn}
We say that $X$ supports a \p-Poincar\'e inequality if there  are constants $C,\lambda\ge 1$ such that
whenever $u$ is a function on $X$ with upper gradient $g$ on $X$ and $B$ is a ball in $X$,
\[
  \frac{1}{\mu(B)}\, \int_B|u-u_B|\, d\mu\le C\, \text{rad}(B)\, \left(\frac{1}{\mu(\lambda B)}\, \int_{\lambda B}g^p\, d\mu\right)^{1/p}.
\]
Here $u_B$ denotes the integral average of $u$ on $B$:
\[
  u_B:=\frac{1}{\mu(B)}\, \int_B\, u\, d\mu.
\]
Furthermore, we say that the measure $\mu$ on $X$ is doubling if there is a constant $C\ge 1$ such that
whenever $B$ is a ball in $X$,
\[
   \mu(2B)\le C\, \mu(B).
\]
\end{defn}

\begin{Assumption}\label{Xgeo} 
\emph{Henceforth, in this paper we will assume that $\mu$ is doubling and that $X$ supports a \p-Poincar\'e inequality.}
We refer the interested reader to~\cite{HaKo} for an in-depth discussion on Poincar\'e inequalities. It was
also shown in~\cite{HaKo} that if $X$ is complete, $\mu$ is doubling, and $X$ supports a \p-Poincar\'e 
inequality, then $X$ is quasiconvex.  Given that the notions of prime ends, rectifiability of curves,
and the metric topology are preserved under biLipschitz change in the metric, 
\emph{henceforth we will 
assume also that $X$ is a geodesic space}.
\end{Assumption}


\begin{defn}\label{zero-bound}
Given a domain (open connected set) $\Om\subset X$, the space of Newtonian functions with zero boundary
values is the space
\[
 N^{1,p}_0(\Om):=\{u\in N^{1,p}(X)\, :\, u=0\text{ in }X\setminus\Om\}.
\]
We refer the reader to~\cite{Sh-har} and the references therein for properties related to this function space.
Given a function $u$ defined only on $\Om$, we say that $u\in N^{1,p}_0(\Om)$ if the zero-extension of 
$u$ lies in $N^{1,p}_0(\Om)$.
\end{defn}

Finally, we introduce the concept of \p-minimizers.

\begin{defn}
A function $u\in\Np(\Om)$ is said to be a \emph{\p-minimizer in $\Om$} if it has minimal \p-energy in 
$\Om$. That is, for all $\phi\in\Np_0(\Om)$,
\[
\int_{\supp(\phi)}g_u^pd\mu\leq\int_{\supp(\phi)}g_{u+\phi}^pd\mu.
\]
Here, $g_u$ and $g_{u+\phi}$ denote the minimal $p$-weak upper gradient of $u$ and $u+\phi$ respectively.
A function that satisfies this condition for nonnegative $\phi\in\Np_0(\Om)$ is said to be 
a \emph{\p-superminimizer in $\Om$}. A function is said to be \emph{\p-harmonic in $\Om$} if it is a 
continuous \p-minimizer in $\Om$.
\end{defn}

As the results in Kinnunen-Shanmugalingam \cite{KS} show, 
under the hypotheses considered in this paper, every 
\p-minimizer can be modified on a set of $p$-capacity zero to obtain a locally H\"older continuous 
\p-harmonic function.

The lower semicontinuous regularization of a function $u$ is  
\[
u^*(x)=\essliminf_{y\to x}u(y).
\]

As shown in~\cite{KM}, the equality $u^*=u$ holds outside a set of zero \p-capacity when $u$ is a
 \p-superminimizer. 
For this reason, any \p-superminimizer used in this paper will be 
assumed to be lower semicontinuously regularized in this manner.
Recall from the above Definition~\ref{zero-bound} that a function
defined on $\Om$ is in $N^{1,p}_0(\Om)$ if its zero-extension to $X\setminus\Om$ 
is in $N^{1,p}(X)$.

\begin{defn}\label{def-obst}
Let $V\subset X$ be open and bounded, with $C_p(X\sm V)>0$. Then, for $f\in\Np(V)$ and $\psi:V\to\eR$, we define the set
\[
\mathcal{K}_{\psi,f}(V):=\{v\in\Np(V)\ :\ v-f\in\Np_0(V),\ v\geq\psi\text{ a.e. in }V\}.
\]
A function $u\in\mathcal{K}_{\psi,f}(V)$ is said to be a \emph{solution of the $\mathcal{K}_{\psi,f}(V)$-obstacle problem} if
\[
\int_V g_u^pd\mu\leq\int_Vg_v^pd\mu,\text{ for all $v\in\mathcal{K}_{\psi,f}(V)$}.
\]
\end{defn}

It is shown in \cite[Theorem~3.2]{KM} that solutions to the $\mathcal{K}_{\psi,f}(V)$-obstacle problem 
exist and are unique (in $\Np(V)$), provided $\mathcal{K}_{\psi,f}(V)\not=\emptyset$.


Given a function $f\in N^{1,p}(X)$ and a bounded domain $\Om\subset X$ with $C_p(X\setminus\Om)>0$, there is a 
unique function $u\in N^{1,p}(X)$ such that $u-f\in N^{1,p}_0(\Om)$ and $u$ 
is \p-harmonic in $\Om$. We denote this solution $u$ by
$H_\Om f$. See~\cite{Sh-har} for the proofs of existence and uniqueness of such solutions.
Note that $H_\Om f$ is the solution to the $\mathcal{K}_{-\infty,f}(\Om)$-obstacle problem.
The condition $C_p(X\setminus\Om)>0$ is needed in order to have non-trivial solutions in $\Om$. Should
$C_p(X\setminus\Om)=0$, then $N^{1,p}_0(\Om)=N^{1,p}(X)$, and in this case for every non-negative 
$f\in N^{1,p}_0(\Om)$ we would have that $H_\Om f$ be a non-negative $p$-harmonic function on $X$ itself,
and hence by the Harnack indequality (see~\cite{KS}) we would have $H_\Om f=0$. By assuming 
$C_p(X\setminus\Om)>0$ we avoid this problem.

Our setting in this paper will primarily be $\overline{\Om}^P$. 
%
Since the subspace topology of $\Om$ inherited from $\overline{\Om}^P$ agrees with the standard
metric topology on $\Om$ inherited from $X$, 
the Newton-Sobolev space $N^{1,p}(\Om)$ can be seen as the function space corresponding to both 
$\Om$, seen as a domain in $X$, and $\Om$, seen as a domain in $\overline{\Om}^P$.
The restriction of $f\in N^{1,p}(X)$ to $\Om$ belongs to $N^{1,p}(\Om)$, and thus 
the notation $Hf:=H_\Om f$ is unambiguous.

However, one should keep in mind that in general functions in $N^{1,p}(X)$, when restricted to $\Om$, may
not have a natural extension to $\bdy_P\Om$. 
This is in contrast to the standard boundary $\bdy\Om$, where one can consider traces of Sobolev functions
as discussed for example in~\cite{Maz} and~\cite{HM}.
In this paper we use $Hf$ for such $f$ only as an intermediate
tool to study the Perron solutions adapted to $\bdy_P\Om$, but not as the end product itself.

\section{Structure of the end and prime end boundaries}

In this section we discuss some structures of the prime end boundary; these structures are useful in the 
subsequent sections where we consider the Perron method for the prime end boundary of a bounded domain.
We first state 
two elementary lemmas regarding the geometry of chains. The proof of these lemmas use the properness
of $X$ (that is, closed and bounded subsets of $X$ are compact).

\begin{lem}\label{lem:sep2}
Given a chain $\{E_k\}_k$, for every $\eps>0$ there is an acceptable set $E_j\in\{E_k\}_k$ such that
\[
E_j\subset N\Big(I(\{E_k\}_k),\eps\Big).
\]
\end{lem}


\begin{lem}\label{lem:sep3}
Let $\{E_k\}_k$ be a chain. Then, for every $\eps>0$ and integer $k$, there is a connected component $C_k^\eps$ of $N(I(\{E_k\}_k),\eps)\cap E_k$ such that $I(\{E_k\}_k)\subset \overline{C_k^\eps}$.
\end{lem}


The above lemmas can be proven by direct topological arguments and by using the definition of ends; we leave the proof
to the interested reader.
Next we prove two useful lemmas about the topology on $\OmPC$.

\begin{lem}\label{lem:l1}
If $\{x_k\}_k$ is a sequence of points in $\Om$ and $[\{E_k\}_k]\in\bdy_E\Om$ such that
$x_k\to [\{E_k\}_k]$, then no subsequence of $\{x_k\}_k$ has a limit point in $\Om$. 
\end{lem}

\begin{proof}
Note that $\bigcap_k\overline{E_k}\subset\bdy\Om$
and, for each positive integer $j$, the tail-end of the sequence $\{x_k\}_k$ lies in $E_j$. Therefore, every cluster point of
$\{x_k\}_k$ must lie in $\bigcap_k\overline{E_k}\subset\bdy\Om$.
\end{proof}

\begin{lem}\label{lem:nbhd-bdy}
If $U\subset\OmPC$
is an open set in the prime end topology 
such that $\bdy_P\Om\subset U$, then for each 
$[\{E_k\}_k]\in\bdy_P\Om$ and for each $\{E_k\}_k\in[\{E_k\}_k]$, there is a positive integer $k_U$ such that
$E_{k_U}\subset U$.
\end{lem}

\begin{proof}
We prove this lemma by contradiction. Suppose that $\{E_k\}\in[\{E_k\}_k]\in\bdy_P\Om$ 
such that for each positive integer $k$ we have $E_k\not\subset U$, that is, we can find
$x_k\in E_k\setminus U$. It then follows that $\{x_k\}_k$ is a sequence in $\Om$ with
$x_k\to [\{E_k\}_k]$.  But then, because $U$ is open in the sequential topology of
$\Om\cup\bdy_P\Om$ and $[\{E_k\}_k]\in U$, we must necessarily have a positive integer $k_U$ 
such that whenever $k\ge k_U$,
$x_k\in U$, which contradicts the choice of $x_k\in E_k\setminus U$.
\end{proof}

Next, we prove a useful relation between $\bdy_{SP}\Om$ and $\bdy_P\Om$.

\begin{thm}\label{spdense}
With respect to the prime end topology on $\OmPC$, $\bdy_{SP}\Om$ is dense in $\bdy_P\Om$.
\end{thm}

\begin{proof}
As in Assumption~\ref{Xgeo}, we assume that $X$ is a geodesic space.

Given a prime end $[\{E_k\}_k]\in\bdy_P\Om\sm\bdy_{SP}\Om$, fix a representative chain $\{E_k\}_k$
of $[\{E_k\}_k]$ 
such that $E_n\subset N(I[\{E_k\}_k],\frac{1}{n})$. Choose a sequence $\{x_n\}_n$ in $\Om$ such that
$x_n\in E_n$ for each positive integer $n$.

For each $x_n$, let $R_n=\text{dist}(x_n, X\setminus\Omega)$ 
and pick $y_n\in\overline{B(x_n,R_n)}\cap\bdy\Om$. 
Note that, since $x_n\in N(I[\{E_k\}_k],\frac{1}{n})$, we have that $R_n\leq\frac{1}{n}$.

Since $X$ is a geodesic space and $B(x_n,R_n)\subset\Om$, there is a geodesic $\ga_n:[0,R_n]\to\overline{\Om}$ from 
$x_n$ to $y_n$ such that $\ga_n([0,R_n))\subset B(x_n,R_n)\subset\Om$. 
Therefore, $y_n$ is accessible and there is a prime end 
$[\{F_k^n\}_k]\in\bdy_{SP}\Om$ such that $I[\{F_k^n\}_k]=\{y_n\}$ and $\ga_n$ accesses $y_n$ through 
$[\{F_k^n\}_k]$ (see Definition~\ref{accthrough}). 
Though not relevant at the moment, for future use in the proof of Proposition \ref{ctsres} we note 
that, since $B(x_n,R_n)$ is connected, $d_M(x_n,[\{F_k^n\}_k])=R_n\leq\frac{1}{n}$. Furthermore, we can
choose $F_k^n$ so that $\text{diam}(F_k^n)\le 1/k$.

We now prove that $[\{F_k^n\}_k]\peto[\{E_k\}_k]$. Suppose this is not the case. Then 
there is a positive integer $K$ such that, for each positive integer $n$, there is an integer $j_n\ge n$
so that for each positive integer $k$ we can find a point $z_{j_n}\in F_k^{j_n}\setminus E_K$. 
The choice of $z_{j_n}$ does indeed depend on $k$ as well, but since we next fix a choice of positive integer $k$,
we do not indicate the dependance of $z_{j_n}$ on $k$ in the notation. Indeed, we now
choose $k\ge 2K+2n$. 

On the other hand, for $n\ge 2K$
we have $x_{j_n}\in E_{K+1}\cap\gamma_{j_n}$, and a  set $\beta_{j_n}=\gamma_{j_n}\cup F_k^{j_n}$,
containing $x_{j_n}$ and $z_{j_n}$, with 
diameter $\text{diam}(\beta_{j_n})\le 1/n+1/k\le 2/n$.  We now show that $\beta_{j_n}$ is connected.
We do not claim here that 
$x_{j_n}\in F_k^{j_n}$, but note that a point in $\gamma_{j_n}$ lies in $F_k^{j_n}$, and
a compact subcurve of $\gamma_{j_n}$ therefore connects $x_{j_n}$ to this point. Hence 
$\beta_{j_n}$ is connected.  Thus we have a point
$z_{j_n}\in\beta_{j_n}$ that lies outside $E_K$, and a point $x_{j_n}\in E_{K+1}\cap\beta_{j_n}$. It follows that
$\text{dist}_M(\Om\cap\bdy E_K,\Om\cap\bdy E_{K+1})\le 2/n$ for sufficiently large $n$. Letting $n\to \infty$ we 
obtain that $\text{dist}_M(\Om\cap\bdy E_K,\Om\cap\bdy E_{K+1})=0$,
which violates the definition of a chain. Hence
we know that $[\{F_k^n\}_k]\peto[\{E_k\}_k]$, completing the proof of the theorem.

\end{proof}

The next lemma provides a connection between locally rectifiable curves of infinite length and ends that are, in some
sense, from the side of those curves.

\begin{lem}\label{lem:Ass(a)}
Let $\Om$ be a bounded domain in $X$.
Suppose that $\ga$ is a curve in $\Om$ such that 
\[
  I(\ga):=\bigcap_{n\in\N}\overline{\ga((n,\infty))}\subset\bdy\Om,
\]
and set
\[
E(\ga):=\{[\{F_k\}_k]\in\bdy_E\Om\ :\ \forall k\in\N,\ \exists t_k\text{ such that }\ga([t_k,\infty))\subset F_k\}.
\]
As in Assumption~\ref{St-Assume} below we consider the order relation $\le$ on $E(\ga)$ defined by $x\le y$ if and only if 
$x|y$. Then $E(\ga)$ has a minimal (or, least) element $[\{E_k\}_k]$. Furthermore, for each $[\{F_k\}_k]\in E(\gamma)$
we have $[\{E_k\}_k]$ divides $[\{F_k\}_k]$ and $I[\{E_k\}_k]=I(\ga)$.
\end{lem}

\begin{proof}
For each positive integer $k$ let $E_k$ denote the 
connected component of the set 
$N_M(\ga((k,\infty)),1/k)\cap\Om$ that contains the tail-end $\ga((k,\infty))$ of $\ga$. 
We will show that the end corresponding to the chain $\{E_k\}_k$ should be a 
minimal end in $E(\ga)$.

It is easily seen that $[\{E_k\}_k]\in E(\ga)$. So it suffices to show that whenever $[\{F_k\}_k]\in E(\ga)$,
the end $[\{E_k\}_k]$ divides $[\{F_k\}_k]$. To do so, let $[\{F_k\}_k]\in E(\ga)$.  
We want to show that given a positive integer $k$ there is a positive
integer $j_k$ such that $E_{j_k}\subset F_k$.

Suppose that the above is not true. Then for each positive integer $j$ the set $E_j\setminus F_k$ is non-empty.
By the construction of $E_j$,
for any $x\in E_j\sm F_k$ there is a real number $t_x\in[j,\infty)$ with $d_M(x,\ga(t_x))<1/j$. 
Note that by the definition of chains, $\dist_M(\bdy F_k\cap\Om,\bdy F_{k+1}\cap\Om)>0$. So we can
choose a positive integer $J$ such that
$1/J<\dist_M(\bdy F_k\cap\Om,\bdy F_{k+1}\cap\Om)$. Consider $j\geq J$, and fix $x_j\in E_j\sm F_k$,
and set $t_j:=t_{x_j}$. We then have 
\[
  d_M(x_j,\gamma(t_j))<1/j\le 1/J<\dist_M(\bdy F_k\cap\Om,\bdy F_{k+1}\cap\Om),
\]
It follows now from the fact that $x_j\not\in F_k$ that $\gamma(t_j)\not\in F_{k+1}$. Consequently, for each
positive integer $j>J$ we can find a real number $t_j\ge j$ such that $\gamma(t_j)\not\in F_{k+1}$. Thus
no tail end of $\gamma$ can lie in $F_{k+1}$, which violates the fact that 
$[\{F_k\}_k]\in E(\gamma)$.

Hence we can conclude that necessarily there is some positive integer $j_k$ such that $E_{j_k}\subset F_k$, that is, the end
$[\{E_k\}_k]$ divides $[\{F_k\}_k]$, concluding the proof.
\end{proof}



In addition to our previous assumptions on $X$, we also will assume for the
remainder of the paper that the domain $\Om$ fulfills the following property:

\begin{Assumption}\label{St-Assume}
\emph{
For every collection $\mathcal{F}$ of ends that is totally ordered by division such that
$x\le y$ if and only if $x|y$, there is an end $[\{G_k\}_k]$ 
such that $[\{G_k\}_k]\le [\{F_k\}_k]$ for every $[\{F_k\}_k]\in\mathcal{F}$.
}
\end{Assumption}

The above assumption essentially states that we assume that the collection of all ends of $\Om$ 
satisfies the hypotheses of Zorn's lemma. 

Should $\Om$ be a simply connected bounded planar domain, the above condition is seen to hold true. The proof
of this fact there goes through the Riemann mapping theorem; in more general settings it is not clear to us 
whether the above condition automatically holds. However, in many situations this condition is directly
verifiable. If $\bdy_{SP}\Om$ is compact, then by Theorem~\ref{spdense} we know that $\bdy_P\Om=\bdy_{SP}\Om$,
and in this case
the fact that above assumption holds is a consequence of the results
found in~\cite[Section~7]{ABBS}. Indeed, by the results in~\cite{ABBS}, it follows that given an end $[\{E_k\}_k]$,
every point in $I[\{E_k\}_k]$ is accessible through $[\{E_k\}_k]$ by rectifiable curves, and hence a prime end 
from $\bdy_{SP}\Om$ divides $[\{E_k\}_k]$.

Under the assumption of~\ref{St-Assume}, 
we have the following fact about $\overline{\Om}^E$.

\begin{thm}\label{thm-E1}
Suppose that $\Om$ is a bounded domain
satisfying the assumption given in Definition~\ref{St-Assume}.
Let $[\{E_k\}_k]$ be an end of $\Om$. Then there is a prime end $[\{F_k\}_k]$ of $\Om$ that divides
$[\{E_k\}_k]$.
\end{thm}

\begin{proof}
Consider the set $\mathcal{E}$ of ends that divide $[\{E_k\}_k]$, ordered by division. If this set contains only 
$[\{E_k\}_k]$, then $[\{E_k\}_k]$ is a prime end.

Assume $\mathcal{E}$ has more than one element. Let $\mathcal{F}$ be a totally ordered subset 
of $\mathcal{E}$, indexed by a corresponding totally 
ordered set $A$. By the assumption given in the theorem, there is an element 
$[\{G_k\}_k]$ that 
divides all the elements of $\mathcal{F}$. Since each of these elements divides $[\{E_k\}_k]$ in turn, 
$[\{G_k\}_k]\Big|[\{E_k\}_k]$. Thus, $[\{G_k\}_k]\in\mathcal{E}$, satisfying the conditions for the use of 
Zorn's Lemma. Thus $\mathcal{E}$ has a minimal element, and this minimal element is necessarily a prime end.
\end{proof}

Finally, we prove the following consequence of the assumptions made on $\Om$ earlier in  
this section. This will be integral to our results in the next section.

\begin{lem}\label{lem:pega}
Let $\Om$ be a bounded domain satisfying the assumption given in Defintion~\ref{St-Assume},
and let $\ga$ be a curve in $\Om$ such that $\ds\bigcap_{n\in\N}\overline{\ga((n,\infty))}\subset\bdy\Om.$ 
Then there is a prime end $[\{A_k\}_k]$ such that $[\{A_k\}_k]\Big|[\{F_k\}_k]$ for every 
$[\{F_k\}_k]\in E(\ga)$, and $A_k\cap\ga\not=\emptyset$ for every integer $k$.
\end{lem}

One cannot in general expect this prime end to be in $E(\ga)$, as the harmonic comb example shows.
Thus the best possible link the prime end has to $\ga$ is the condition that 
$A_k\cap\ga\not=\emptyset$ for every integer $k$.

\begin{proof}
By Lemma~\ref{lem:Ass(a)}, 
we know that $E(\ga)$ has a minimal element $[\{G_k\}_k]$ and that 
$[\{G_k\}_k]\Big|[\{F_k\}_k]$ for every $[\{F_k\}_k]\in E(\ga)$.

If $[\{G_k\}_k]$ happens to be a prime end, then set $[\{A_k\}_k]=[\{G_k\}_k]$ and the proof would be complete. 
If not, we may use Theorem~\ref{thm-E1} to obtain a prime end $[\{H_k\}_k]$ that divides $[\{G_k\}_k]$. Since 
$[\{G_k\}_k]$ is minimal in $E(\ga)$, it must be the case that $[\{H_k\}_k]\not\in E(\ga)$. 
Now we have one of two following possibilities:
\begin{enumerate}
\item For every $k\in\N$, there is a positive real number $t_k$ such that $\ga(t_k)\in H_k$.
\item There exists a positive integer $k_0$ such that $H_{k_0}\cap\ga=\emptyset$.
\end{enumerate}

If $[\{H_k\}_k]$ behaves as in (a), we simply take $[\{A_k\}_k]=[\{H_k\}_k]$ and 
the proof is complete. We now show 
that possibility~(b) does not occur. 
We may also, without loss of generality, suppose that $\overline{H_k}\cap\Om=H_k$.

Assume that $[\{H_k\}_k]$ behaves as in (b). For simplicity, we may take $k_0=1$. Then, define
\[
m_H:=\dist_M(\bdy H_1\cap\Om,\bdy H_2\cap\Om)
\]
and
\[
\widehat{H}_k:=\bigg(\bigcup_{x\in H_2}B_M(x,(1-\frac{1}{k+1})m_H)\bigg)\cap H_1.
\]
Then $H_2\subset \widehat{H}_k\subset \widehat{H}_{k+1}\subset H_1$ and
\[
\dist_M(\bdy\widehat{H}_k\cap\Om,\bdy\widehat{H}_{k+1}\cap\Om)>0
\]for every $k\in\N$. Because $\gamma$ does not intersect $H_1$ and $H_1$ is relatively closed in $\Om$ by
assumption, we have that $\gamma\subset\Om\setminus\overline{H_1}$.

Finally, we define $D_k$ as the connected component of $G_k\sm\overline{\widehat{H}_k}$ that contains 
the tail end of $\gamma$ inside $G_k$.
Since $\overline{\widehat{H}_k}$ contains no points of $\ga$, we know that this component exists. 
By construction, $D_k\supset D_{k+1}$ and $\bigcap_{k\in\N}\overline{D_k}\subset \bdy\Om$. As before, 
we need only show that $\dist_M(\bdy D_k\cap\Om,\bdy D_{k+1}\cap\Om)>0$ for all $k\in\N$ to establish 
that $\{D_k\}_k$ is a chain.
Let
\[
M_k=\min\{\dist_M(\bdy \widehat{H}_k\cap\Om,\bdy \widehat{H}_{k+1}\cap\Om), \dist_M(\bdy G_k\cap\Om,\bdy G_{k+1}\cap\Om)\}.
\] 
Note that $M_k>0$. Take $x\in\bdy D_k\cap\Om$ and $y\in\bdy D_{k+1}\cap\Om$ and consider the following cases.

\noindent{\bf Case 1:} $x\in\bdy G_k\cap\Om$ and $y\in\bdy G_{k+1}\cap\Om$. In this case, we 
immediately have that $d_M(x,y)\geq M_k$.

\noindent{\bf Case 2:} $x\in\bdy G_k\cap\Om$ and $y\in\bdy \widehat{H}_{k+1}\cap\Om$, but 
$y\not\in\bdy G_{k+1}\cap\Om$. Here, it must be the case that $y\in G_{k+1}$. So $d_M(x,y)\geq M_k$.

\noindent{\bf Case 3:} $x\in\bdy \widehat{H}_k\cap\Om$ and $y\in\bdy \widehat{H}_{k+1}\cap\Om$.
As in Case 1, we immediately have that $d_M(x,y)\geq M_k$.

\noindent{\bf Case 4:} $x\in\bdy \widehat{H}_k\cap\Om$ and $y\in\bdy G_{k+1}\cap\Om$, but 
$y\not\in\bdy \widehat{H}_{k+1}\cap\Om$. Here, it must be that $y\in G_k\sm\overline{\widehat{H}_{k+1}}$. 
So $d_M(x,y)\geq M_k$.

\noindent{\bf Case 5:} $x\not\in\Om\cap(\bdy G_k\cup\bdy\widehat{H}_k)$
or $y\not\in\cap(\bdy G_{k+1}\cup\bdy\widehat{H}_{k+1})$. We will focus on the first possibility, the second
being handled in a very similar manner. Since $x\not\in\Om\cap(\bdy G_k\cup\bdy\widehat{H}_k)$, it follows
that $x$ is in the interior of $G_k$, and hence in the interior of $G_k\setminus\overline{\widehat{H}_k}$.
It follows that for sufficiently small $r>0$ the connected set $B_M(x,r)\subset G_k\setminus\overline{\widehat{H}_k}$,
which then means that $B_M(x,r)\subset D_k$, violating the fact that $x\in\bdy D_k$. Hence this case is not possible.

The above argument allows us to conclude that $\dist_M(\bdy D_k\cap\Om,\bdy D_{k+1}\cap\Om)>0$, and so 
$\{D_k\}_k$ is a chain and $[\{D_k\}_k]$ is an end. By construction, $[\{D_k\}_k]\Big|[\{G_k\}_k]$ and 
$[\{D_k\}_k]\in E(\ga)$.  Because $[\{G_k\}_k]$ divides each end in $E(\gamma)$, and so 
$[\{D_k\}_k]=[\{G_k\}_k]$.
However, by construction $[\{H_k\}_k]$ does not divide $[\{D_k\}_k]$, which
violates the choice of $[\{H_k\}_k]$ as a prime end that divides $[\{G_k\}_k]$. 
Hence the alternative~(b) cannot occur. 
This completes the proof of the lemma.
\end{proof}

\section{Prime ends are from all sides.}
\label{sect-3}

The goal of this section is to show that if $V\subset\Om$ is an open connected set such that
$\bdy V\cap\bdy\Om$ is non-empty, then there is a prime end from the side of $V$. To do so we 
employ the inner metric $\dinn^V$ (see Definition \ref{defn:dinn}). 

For each $\eps>0$ let $V_\eps:=\{x\in V\, :\, \dist(x,X\setminus V)>\eps\}$, and for a 
(locally rectifiable) curve $\ga$ of infinite length in $X$, let 
\[
    I(\ga):=\bigcap_{n\in\N}\overline{\ga((n,\infty))}.
\]
Note that if $\ga\subset\overline{V}$, then
$I(\ga)$ is a connected compact subset of $\overline{V}$.

\begin{lem}\label{lem:leave-compacts}
 Let  $V\subset\Om$ be an open connected set and suppose that $x_\infty\in \bdy V\cap\bdy\Om$. Let
 $\{x_k\}_k$ be a sequence of points in $V$ such that $\lim_kx_k=x_\infty$, and let $x_0\in V$. Suppose that for each
 positive integer $k$ the $\dinn^V$-geodesic $\ga_{x_0,x_k}^V$ does not intersect $\bdy\Om$. Then 
 there is a curve $\ga:[0,\infty)\to\overline{V}$ such that $\ga$ is the local uniform limit of a 
 subsequence of the sequence of curves $\{\ga_{x_0,x_k}^V\}$. Furthermore, if $\ga$ has infinite length, 
 then $I(\ga)\subset\bdy V$.
\end{lem}

\begin{rem}\label{rem:single-access}
Note that if there are two points $z,w\in V$ such that the $\dinn^V$-geodesic connecting $z$ to $w$ intersects
$\bdy\Om$, then, because this geodesic has finite length (with respect to the metric $d$), it follows that
there is a point $x_0\in\bdy V\cap\bdy\Om$ that is accessible from the side of $V$. See Definition~\ref{accside}
for the definition of ``accessibility from the side" of $V$. As a consequence, if such points $z,w$ exist, then
there is a prime end from the side of $V$, and we can choose this prime end from the class $\bdy_{SP}\Om$.
\end{rem}

\begin{proof}[Proof of Lemma~\ref{lem:leave-compacts}]
The existence of the curve $\ga$ is easily given by applying the Arzel\`{a}-Ascoli theorem to the 
equibounded (since $\Om$ is bounded) equicontinuous (since these curves are $1$-Lipschitz maps with respect
to the underlying metric $d$) family 
$\{\ga_{x_0,x_k}^V\}$. It is also clear that $I(\ga)\subset\overline{V}$ and that each subcurve of 
$\ga$ is a $\dinn^V$-geodesic between its endpoints. Demonstrating that when $\ga$ has infinite length
$I(\ga)\subset\bdy V$ requires  slightly more work.

We argue by contradiction. Suppose that there is a point $y\in I(\ga)\cap V$, 
and pick a sequence $\{y_i\}$ with $y_i\to y$ and 
$y_i\in\ga([i,\infty))$ for each $i$. 
Since $V$ is open, there is a sufficiently small 
neighborhood of $y$ within $V$ such that the metrics $d$ and $\dinn^V$ are biLipschitz equivalent
inside this neighborhood. 
Thus, $y_i$ converges to $y$ with respect to $\dinn^V$, requiring that $\dinn^V(y_i,y)$ be uniformly 
bounded by some $M< \infty$.

Recall that $\dinn^V(x_0,y)$ must be finite; denote this quantity by $N$. By the triangle inequality, 
\[
\dinn^V(x_0,y_i)\leq \dinn^V(x_0,y)+\dinn^V(y_i,y)\leq N+M. 
\]
Since $M$ and $N$ are independent of $i$, 
we have that $\dinn^V(x_0,y_i)$ is uniformly bounded. But we picked $y_i$ such that $y_i\in\ga([i,\infty))$, and since 
$\ga$ is locally a geodesic
with infinite length, $\dinn^V(x_0,y_i)\geq i$ for each $i$. But this contradicts the above 
bound on $\dinn^V(x_0,y_i)$. Thus, we have that $y\not\in V$, that is,
$I(\ga)\subset \bdy V$.
\end{proof}

\begin{thm}\label{main1}
Let $\Om$ be a bounded connected open set satisfying the condition given in
Assumption~\ref{St-Assume}, and $V\subset\Om$ be an open, connected set. If $\bdy\Om\cap\bdy V$
is non-empty, then there is a prime end of $\Om$ from the side of $V$.
\end{thm}

\begin{proof}
 If there is a rectifiable curve in $\overline{V}$ that connects a point in $V$ to a point in $\bdy V\cap\bdy\Om$, then
 the accessibility results of~\cite{ABBS} gives a corresponding prime end from the side of $V$. 
 See also Remark~\ref{rem:single-access} above.  Hence, without loss of
 generality, we may assume that there is no rectifiable curve in $\overline{V}$ that connects some point in $V$ to
 a point in $\bdy\Om\cap\bdy V$. Note that we now fulfill the assumptions of Lemma~\ref{lem:leave-compacts}, allowing its
 use in the remainder of the proof.
 
We fix $x_0\in V$ and $x_\infty\in\bdy\Om\cap\bdy V$, and let $\{x_k\}_k$ be a sequence of points in $V$ such that 
$\lim_k x_k=x_\infty$. For each positive integer $k$ let $\ga_{x_0,x_k}^V$ be as in the statement of 
Lemma~\ref{lem:leave-compacts}. Clearly $\ga_{x_0,x_k}^V$ cannot intersect $\bdy\Om$ because if it does, then 
we have a point in $\bdy V\cap\bdy\Om$ that is accessible from $V$, violating the assumption stated in the 
previous paragraph of this proof; see Remark~\ref{rem:single-access}. Hence
Lemma~\ref{lem:leave-compacts} gives us a locally 
rectifiable curve $\ga:[0,\infty)\to\overline{V}$ such that $\ga(0)=x_0$, and for each $t>0$ the curve 
$\ga|_{[0,t]}$ is a $\dinn^{\overline{V}}$--geodesic that lies inside $\Om$. 
Since we assumed that there are no rectifiable curves connecting a point in $V$ to
$\bdy V\cap\bdy \Om$, $\ga$ must have infinite length. Thus, $I(\ga)\subset\bdy V$.

Now the proof diverges according to two possibilities.

\noindent {\bf Case 1:} $I(\ga)\subset\bdy\Om\cap\bdy V$. Then we can proceed to construct an end as follows.  
For $k\in\N$ we set $E_k$ to be the connected component of $N(I(\ga),1/k)\cap \Om$ that contains 
$\ga([t_k,\infty))$ for some $t_k>0$.
Each $E_k$ is an acceptable set, and $\{E_k\}_k$ satisfies
the conditions of a chain. Note that 
\[
\dist_M(\Om\cap\bdy E_k, \Om\cap\bdy E_{k+1})\ge \dist(\Om\cap\bdy E_k, \Om\cap\bdy E_{k+1})\ge \tfrac{1}{k(k+1)}>0
\]
and that $I[\{E_k\}_k]=I(\ga)\subset\bdy\Om\cap\bdy V\subset\bdy\Om$.

It is clear that $[\{E_k\}_k]$ is from the side of $V$, however there is no {\it a priori} reason for $[\{E_k\}_k]$ to be prime. Note, however,
that $[\{E_k\}_k]$ is clearly a member of $E(\ga)$, and so by Lemma~\ref{lem:pega} there is a prime end $[\{A_k\}_k]$ dividing $[\{E_k\}_k]$ such that
$A_k\cap\ga\not=\emptyset$ for each $k$. Therefore, for each $k$, $A_k\cap\overline{V}$ must be nonempty. 
Furthermore, we can choose each $A_k$ to be open (see~\cite[Remark~4.5]{ABBS}), and so we conclude that
$A_k\cap V$ is nonempty. It follows that $[\{A_k\}_k]$ is a prime end from the side of $V$.

\noindent {\bf Case 2:} $I(\ga)\not\subset\bdy\Om\cap\bdy V$. 
We denote by $\overline{B}(x,r)$ the closed ball $\{y\in X\, :\, d(x,y)\le r\}$ rather than the closure of the open ball 
$B(x,r)$. Recall from Assumption~\ref{Xgeo} that $X$ is a
geodesic space. 
It follows that whenever $x\in X$ and $r>0$, each pair of points $z,w\in \overline{B}(x,r)$
can be connected in the closed ball $\overline{B}(x,r)$ by a curve of length at most $2r$. Again, because $X$ is
doubling and hence separable, we can cover $\partial V\setminus\partial\Om$ by at most a countable family of balls
$B(z_i,r_i)$ with $z_i\in \bdy V\setminus\bdy\Om$ and $r_i=\min\{\dist(z_i,X\setminus\Om), d(z_i,x_0)\}/10$. Setting
\[
    V_j:=V\cup\bigcup_{i=1}^j B(z_i,r_i)
\]
for positive integers $j$, note that if $x,y\in V$, then 
\begin{equation}\label{eq:inners}
   d(x,y)\le \dinn^{V_{j+1}}(x,y)\le \dinn^{V_j}(x,y)\le \dinn^V(x,y).
\end{equation}
As in the first part of the proof, we obtain curves $\gamma_j$ for each $j$ that are locally uniform limits
of $\dinn^{V_j}$--geodesics connecting $x_0$ to $x_k$. Because of~\eqref{eq:inners}, and because each
$\ga_n$ is a $\dinn^{V_n}$-geodesic, we know that if $\ga_m(t_j)\in B(z_j,r_j)\cap V$
for some $m\le j$, then for all $n\ge j$ we have that $\ga_n([t_j+2r_j,\infty))$ does not intersect
$\overline{B}(z_j,r_j)$. It follows that for $n\ge j$ we have that $I(\ga_n)\cap \overline{B}(z_j,r_j)$ is empty.
Therefore,
\[
   I(\ga_n)\subset \bdy V\setminus\bigcup_{i=1}^n B(z_i,r_i).
\] 
A final application of Arzel\`{a}-Ascoli theorem gives a subsequence of $\{\ga_n\}_n$ that converges locally uniformly
to a curve $\be:[0,\infty)\to\overline{\bigcup_jV_j}$ such that $\be(0)=x_0$, and because for each $n\in\N$ we have
that $\be([t_n+2r_n,\infty))\cap\overline{B}(z_n,r_n)$ is empty,
\[
     I(\be)\subset\bdy V\setminus\bigcup_{i\in\N}B(z_i,r_i)=\bdy V\cap\bdy\Om.
\]
The proof is now completed by applying the argument at the end of the proof of
Case~1 to $\be$ instead of $\ga$.
\end{proof}

The following corollary to the above theorem gives us a useful fact, namely that compact containment of
connected sets  in 
$\Om$ is the same in both the Prime End topology and the topology on $\Om$ inherited from $X$.
Note however that if we do not require $V$ to be connected, the following theorem would be false in general.

\begin{cor}\label{cor:main}
Let $V\subset\Om$ be an open, connected set. Then $\overline{V}\subset\Om$ if and only if 
$\overline{V}^{P,\Om}\subset\Om$.
\end{cor}

\begin{proof}
If $\overline{V}\subset\Om$, then $\overline{V}=\overline{V}^{P,\Om}$. If $\overline{V}^{P,\Om}\subset\Om$, 
then clearly there can be no prime ends from the side of $V$. Thus, by Theorem~\ref{main1}, 
$\bdy V\cap\bdy\Om=\emptyset$. Therefore, $\overline{V}\subset\Om$.
\end{proof}

\section{Prime End Capacity and Newtonian Spaces}

Recall that in Section~3 a notion of $p$-capacity associated with the space $N^{1,p}(X)$
was discussed. In this section we will modify this notion to take into consideration the
structure of $\Om\subset\overline{\Om}^P$. 
This new version of $p$-capacity is useful in the study of the Perron method adapted to
the prime end boundary of $\Om$.

\begin{defn}\label{def:cbar}
For $E\subset\overline{\Om}^P$ let
\[
\CpP(E)=\inf_{u\in\mathcal{A}_E} ||u||^p_{\Np(\Om)},
\]
where $u\in \mathcal{A}_E$ if $u\in\Np(\Om)$ satisfies both $u\geq 1$ on $E\cap\Om$ and
\[
\liminf_{\Om\ni y\peto x} u(y)\geq 1 \text{  for all $x\in E\cap\bdy_P\Om$}.
\]
\end{defn}

In the above definition, we can impose the additional requirement that $0\le u\le 1$
without any change in the resulting number for $E$.

The capacity $\CpP$ satisfies the usual basic properties of a capacity.

\begin{lem}
Let $E,E_1,E_2,E_3,\ldots$ be arbitrary subsets of $\overline{\Om}$. Then
\begin{enumerate}
\item $\CpP(\emptyset)=0$,
\item $\mu(E\cap\Om)\leq\CpP(E)$,
\item If $E_1\subset E_2$, then $\CpP(E_1)\leq\CpP(E_2)$ (monotonicity),
\item $\CpP\ds\left(\bigcup_{i=1}^\infty E_i\right)\leq\ds\sum_{i=1}^\infty\CpP(E_i)$ (countable subadditivity).
\end{enumerate}
\end{lem}

The proof of the above lemma follows precisely the proof of an analogous result in~\cite[Proposition~3.2]{BBS}, and so is 
omitted here.

%

We say that a function $u$ on $X$ is \emph{$p$-quasicontinuous} (or, quasicontinuous)
on an open set $W\subset X$ if for every $\eps>0$ we can
find an open set $U_\eps\subset X$ such that $u\vert_{W\setminus U_\eps}$ is continuous and
$C_p(U_\eps)<\eps$.

\begin{prop}\label{prop:cppout}
Suppose that the measure on $X$ is doubling and supports a \p-Poincar\'e inequality. 
Then $\CpP$ is an outer capacity, i.e. for all $E\subset\overline{\Om}^P$,
\[
\CpP(E)=\inf_G \CpP(G),
\]where the infimum is taken over all $G\supset E$ that are open in $\OmPC$. 
\end{prop}

While the proof of this proposition is very similar to the proof of the related result~\cite[Proposition~3.3]{BBS},
the situation considered by~\cite{BBS} was simpler in that the boundary of the domain considered there was
the so-called Mazurkiewicz boundary, and so the function $w$ defined there in a manner analogous to the proof
here is easily seen to be admissible in computing the capacity. Here additional arguments were needed, and so
for the convenience of the reader we provide the complete proof here. 

\begin{proof}
By the assumptions on $X$ (doubling property of $\mu$ and the support of a $p$-Poincar\'e inequality)
and by the results in~\cite{Sh-rev} and~\cite{BBS2} , we know that functions in
$N^{1,p}(X)$ and functions in $N^{1,p}(\Om)$ are $p$-quasicontinuous.

%
By the monotonicity of $\CpP$, we obtain the inequality 
$\CpP(E)\leq\inf_G \CpP(G)$ for free. We must work harder for the reverse inequality.

Given $E\subset\OmPC$ and $\eps>0$, we pick a function $u\in\mathcal{A}_E$ with $0\leq u\leq 1$ such that
\[
||u||_{\Np(\Om)}\leq\CpP(E)^{1/p}+\eps.
\]

Since $u$ is  quasicontinuous on $\Omega$, we may also take some open set 
$V\subset\Om$ such that $C_p(V)^{1/p}\leq\eps$ and $u|_{\Om\sm V}$ is continuous. Thus, 
$\{x\in\Om\, :\, u(x)>1-\eps\}\setminus V$
is an open set in $\Om\sm V$ with respect to the subspace topology. Therefore there is another open set 
$U\subset\Om$ such that
\[
U\sm V= 
  \{x\in\Om\, :\, u(x)>1-\eps\}\setminus V\supset(E\cap\Om)\sm V.
\]

Because $C_p(V)\le \eps^p$, we can choose $v\in N^{1,p}(X)$ satisfying
$||v||_{\Np(X)}<2\eps$, $0\le v\le 1$ on $X$, and $v\ge 1$ on $V$. Set 
\[
w=\frac{u}{1-\eps}+v.
\]

Then $w\geq 1$ on $U\cup V$, which is an open set containing $E\cap\Om$. 
Also, for each $[\{E_k\}_k]\in E\cap\bdy_P\Om$, there is a positive integer $K$ such that $u>1-\eps$ on 
$E_K$. Indeed, if not, then we can find a sequence of points $x_k\in E_k$ such that $u(x_k)\le 1-\eps$
but $x_k\peto [\{E_k\}_k]\in E\cap\bdy_P\Om$, a violation of the choice of $u\in\mathcal{A}_E$.

Let
\[
W=U\cup V\cup \bigcup_{[\{E_k\}_k]\in E\cap\bdy_P\Om}(E_K\cup E_K^P),
\]
where 
$E_K^P$ is as defined in Remark~\ref{rem:pebasis}. Then $W\supset E$ is an 
open set in $\OmPC$ and $w\in\mathcal{A}_W$. So
\begin{align*}
\CpP(E)^{1/p}&\leq \inf_G \CpP(G)^{1/p}\leq\CpP(W)^{1/p}\leq||w||_{\Np(\Om^P)}\\
&\leq\frac{1}{1-\eps}||u||_{\Np(\Om^P)}+||v||_{\Np(\Om^P)}\leq\frac{1}{1-\eps}(\CpP(E)^{1/p}+\eps)+2\eps.
\end{align*}
By letting $\eps\to0$, the proof is complete.
\end{proof}

We also restate the definition of quasicontinuity with respect to this capacity.

\begin{defn}
A function $f:\overline{\Om}^P\to\overline{\R}$ is {\it $\CpP$-quasicontinuous} if, for every $\eps>0$, 
there is a relatively open set $U\subset\overline{\Om}^P$ such that $\CpP(U)<\eps$ and $f|_{\overline{\Om}^P\sm U}$ 
is real-valued continuous.
\end{defn}

It is natural for us to try to further relate $\CpP$ to the usual capacity $C_p$. To do this in a 
meaningful way, we would require a method of relating subsets of $\OmPC$ to those in 
$\overline{\Om}$. Since single elements in $\OmPC$ might correspond to large sets in 
$\overline{\Om}$, there is no easy mapping between $\OmPC$ and $\overline{\Om}$ as in 
the case of the Mazurkiewicz boundary in \cite{BBS}. Instead, we introduce the notion of 
the \emph{Prime End Pushforward} of a set $E\subset \overline{\Om}$ in the following way.

\begin{defn}
Given $E\subset X$, the \emph{$\Om$--Prime End Pushforward} of $E$, denoted $P(E)$, is defined as
\[
P(E):=(E\cap\Om) \cup \{[\{E_k\}]\in\bdy_P\Om |\ I[\{E_k\}]\subset E\}.
\]
\end{defn}

It is clear from the definition that if $E\subset F$, then $P(E)\subset P(F)$. Also, this pushforward can 
easily be shown to be ``an open map", in that if $E$ is open in $X$, then $P(E)$ is open in $\OmPC$. Hence, 
if $E\subset\overline{\Om}$ is relatively open, then $P(E)$ is open in $\OmPC$.

With this definition, we have the following Lemma. Recall that we assume the measure on $X$ to be doubling
and support a \p-Poincar\'e inequality.

\begin{lem}\label{lem:cbp}
Let $E\subset X$. Then
\[
\CpP(P(E))\leq C_p(E).
\]
\end{lem}

\begin{proof}
Given any $\eps>0$, we may pick an open set $G\supset E$ in $X$ such that $C_p(G)\leq C_p(E)+\eps/2$. This 
is due to the fact that $C_p$ is an outer capacity (see~\cite[Corollary~1.3]{BBS2} 
or~\cite[Theorem~5.21]{BBbook}).

Let $f\in\Np(X)$ such that $f=1$ on $G$ and $||f||^p_{\Np(X)}<C_p(G)+\eps/2$. Define 
$\tilde{f}:=f|_\Om$. Note that $\tilde{f}\in\Np(\Om)$.

Immediately, if $x\in P(G)\cap\Om=G\cap\Om$, then $\tilde{f}(x)=1$. For 
$x\in P(G)\cap\bdy_P\Om$, we must look at a sequence $\{y_k\}_k$ in $\Om$ converging 
to $x$ in $\OmPC$. Since $P(G)$ is open,  we may assume that $y_k\in P(G)\cap\Om$ 
for each $k$. Then, clearly, $\ds\liminf_{y_k\peto x}\tilde{f}(y_k)\geq 1$. Thus, $\tilde{f}$ 
is an admissible function for the computation of $\CpP(P(G))$. So,
\[
\CpP(P(E))\leq\CpP(P(G))\leq||\tilde{f}||_{\Np(\Om)}^p\leq C_p(E)+\eps.
\]
Letting $\eps\to0$, the proof is completed.
\end{proof}

Finally, in order to compare boundary values of functions on $\OmPC$, we need to 
consider $N^{1,p}_0(\Om)$ as given in Definition~\ref{zero-bound}. 
%
%

The following proposition is analogous to Proposition~5.4 of~\cite{BBS}. The major difference between our situation
here and that of~\cite{BBS} is that there is no continuous map $\Phi:\partial_P\Om\to \partial\Om$, and so the proof of
the following proposition is more complicated than that found in~\cite{BBS}.

\begin{prop}\label{prop:cpqc}
If $f\in\Np_0(\Om)$, then the zero-extension of $f$ to $\bdy_P\Om$ is $\CpP$-quasicontinuous.
\end{prop}

\begin{proof}
Let $f^0$ be the zero extension of $f$ (as a function on $\Om$) to all of $X$. Then $f^0\in\Np(X)$, and so for any $\eps>0$ there is an open set $U_\eps$ in $X$ such that $C_p(U_\eps)<\eps$ and $f^0|_{X\sm U_\eps}$ is continuous. Now let $\hat{f}:\OmPC\to\overline{\R}$ be defined as
\[
\hat{f}(x)=
\begin{cases}
f(x) & \text{ if }x\in\Om, \\
0 & \text{ if }x\in\bdy_P\Om.
\end{cases}
\] 

By Lemma~\ref{lem:cbp} we know that $\CpP(P(U_\eps))<\eps$. We wish to show that 
$\hat{f}\vert_{\overline{\Om}^P\setminus P(U_\eps)}$ is continuous. 
Let $x\in\OmPC\sm P(U_\eps)$ and $\{y_k\}_k$ be a sequence in $\OmPC\sm P(U_\eps)$ such that 
$y_k\peto x$. We wish to check that $\hat{f}(y_k)\to\hat{f}(x)$. 
Since $\hat{f}\vert_{\Om\sm P(U_\eps)}=f\vert_{\Om\sm U_\eps}$ is continuous, we know that
if $x\in \Om$ then the above convergence holds. So
without loss of generality, 
we may consider the following two cases.


\noindent {\bf Case 1:} $y_k\in\bdy_P\Om$ for each $k$, and $x\in\bdy_P\Om$. Since 
$\hat{f}(y_k)=0=\hat{f}(x)$ for all $k$, this case is immediate.

\noindent {\bf Case 2:} $y_k\in\Om$ for each $k$, and $x\in\bdy_P\Om$. Let 
$\{y_{k_i}\}_i$ be a subsequence of $\{y_k\}_k$. Since $I[x]$ is a compact set and 
$\overline{\Om}$ is a compact subset of $X$, 
there is a further subsequence $\{y_{k_{i,j}}\}_j$ such that, for some $x_0\in I[x]$, 
$y_{k_{i,j}}\to x_0$ in the topology of $X$. Since $y_{k_{i,j}}\in X\sm U_\eps$ for each 
$j$ and $X\sm U_\eps$ is a closed set, the limit $x_0$ of $y_{k_{i,j}}$ cannot lie in $U_\eps$.
Therefore 
$x_0\in X\sm(\Om\cup U_\eps)$, $f^0(x_0)=0$ and 
$\hat{f}(y_{k_{i,j}})=f^0(y_{k_{i,j}})\to 0$. Since this happens for all subsequences of 
$\{y_k\}_k$, we conclude that $\hat{f}(y_{k_{i,j}})\to 0=\hat{f}(x)$.

With both possibilities dealt with, we have proved the desired claim.
\end{proof}

It is possible to define a dual notion to the {\it Prime End Pushforward} of set, namely the {\it Prime End Pullback}.
\begin{defn}
Given $F\subset \OmPC$, the \emph{$\Om$--Prime End Pullback} of $F$, denoted $P^{-1}(F)$, is defined as
\[
P^{-1}(F):=(E\cap\Om) \cup \bigcup_{[\{E_k\}]\in F}I[\{E_k\}].
\]
\end{defn}
It is natural to consider how the two notions interact. It can be quickly deduced from their definitions
that if $E\subset X$ and $F\subset \OmPC$, then $P^{-1}(P(E))\subset E$ and $F\subset P(P^{-1}(F))$. 
As the following two examples show, 
equality does not hold in general for either case.

\begin{egs}
If we take $X=\R^2$, with
\[
\Om:=(0,1)^2\sm\bigcup_{n=2}^\infty\{1/n\}\times(0,1/2],
\]
and let $E=[0,1]^2$, we observe that
\[
P^{-1}(P(E))=[0,1]^2\sm\{0\}\times[0,1/2).
\]
Thus, $E\not\subset P^{-1}(P(E))$ in this case.
\end{egs}

\begin{egs}
Letting $X=\R^2$ and let $\Om$ be the slit disk
\[
\Om=B(0,1)\sm[0,1)\times\{0\}. 
\]
Take (recalling Remark~\ref{rem:pebasis})
\[
F=\{(x,y)\in\Om\ |\ y>0\}^P.
\]
Then
$F\subset\OmPC$ consists of the upper half of the slit disk in addition to the prime ends associated with the
'top' part of the slit. It is then easy to see that $P(P^{-1}(F))$ will contain both 'sides' of the slit, and so
$P(P^{-1}(F))\not\subset F$.
\end{egs}

The proof of the following lemma is mutatis mutandis the same as the proof of Lemma~\ref{lem:cbp}. We leave
it to the interested reader to verify this.

\begin{lem}
Given $E\subset\overline{\Om}^P$, we have
\[
  \CpP(E)\le C_p(P^{-1}(E)).
\]
\end{lem}

\section{The Perron solution with respect to Prime Ends}


Now we are ready to consider the following Dirichlet problem: Given
$g:\bdy_P\Om\to\R$, find a function $u$ that is \p-harmonic on $\Om$ and such that 
$u=g$ on $\bdy_P\Om$ in some sense. The method we use to construct possible
solutions to the above problem for certain type of functions $g$ is the 
Perron method, adapted to $\bdy_P\Om$. We continue to assume the standard
assumptions about $X$ (the doubling property of the measure on $X$, and the
validity of a \p-Poincar\'e inequality on $X$), and that $\Om$ is a bounded
domain in $X$ with $C_p(X\setminus\Om)>0$ such that $\Om$ satisfies the 
condition given in Definition~\ref{St-Assume}.

\begin{defn}
A function $u:\Om\to(-\infty,\infty]$ is said to be \emph{\p-superharmonic} if
\begin{enumerate}
\item $u$ is lower semicontinuous,
\item $u$ is not identically $\infty$ on $\Om$,
\item for every nonempty open set $V\Subset\Om$ and all functions $v\in Lip(X)$, 
          if $v\leq u$ on $\bdy V$, then $H_Vv\leq u$ in $V$.
\end{enumerate}
A function $u$ is said to be \emph{\p-subharmonic} if $-u$ is \p-superharmonic.
\end{defn}

We now define the Perron solution with respect to $\overline{\Om}^P$.

\begin{defn}
Given a function $f:\bdy_P\Om\to\overline{\R}$, let $\mathcal{U}_f(\overline{\Om}^P)$ be the set of 
all \p-superharmonic functions $u$ on $\Om$ bounded below such that
\[
\liminf_{\Om\ni y\peto [\{E_n\}_n]}u(y)\geq f([\{E_n\}_n])\text{  for all }[\{E_n\}_n]\in\bdy_P\Om.
\]
We define the \emph{upper Perron solution} of $f$ by
\[
\itoverline{P}_{\overline{\Om}^P}f(x)=\inf_{u\in\mathcal{U}_f(\overline{\Om}^P)}u(x),\ x\in\Om.
\]
Similarly, let $\mathcal{L}_f(\overline{\Om}^P)$ be the set of all \p-subharmonic functions 
$u$ on $\Om$ bounded above such that
\[
\limsup_{\Om\ni y\peto [\{E_n\}_n]}u(y)\leq f([\{E_n\}_n])\text{  for all }[\{E_n\}_n]\in\bdy_P\Om.
\]
We define the \emph{lower Perron solution} of $f$ by
\[
\itunderline{P}_{\overline{\Om}^P}f(x)=\sup_{u\in\mathcal{L}_f(\overline{\Om}^P)}u(x),\ x\in\Om^P.
\]
Note that $\itunderline{P}_{\overline{\Om}^P}f=-\itoverline{P}_{\overline{\Om}^P}(-f)$.
If $\itoverline{P}_{\overline{\Om}^P}f=\itunderline{P}_{\overline{\Om}^P}f$ on $\Om$, then we let 
$P_{\overline{\Om}^P}f:=\itoverline{P}_{\overline{\Om}^P}f$, and $f$ is said to be \emph{resolutive}.
\end{defn}

For the classical formulation of the Perron solution, it is shown in~\cite[Theorem~5.1]{BBSPer} 
that functions $f\in\Np(X)$ are resolutive. We wish to provide a similar result for an 
appropriate class of functions on $\bdy_P\Om$. Due to the potential 
non-compactness of the 
space $\OmPC$, we must first prove that several important results still hold in this space. 
Chief among them is the following comparison principle. An analogous comparison principle, set up
for the Mazurkiewicz boundary in~\cite[Proposition~7.2]{BBS}, is straightforward to prove because
of the assumption in~\cite{BBS} that the Mazurkiewicz boundary $\partial_M\Om$ is compact. 
Here we overcome the lack of compactness of $\partial_P\Om$ with the aid of Corollary~\ref{cor:main}.

\begin{prop}\label{prop:comp}
Assume that $u$ is \p-superharmonic and that v is \p-subharmonic in $\Om$. If
\[
\infty\not=\limsup_{\Om\ni y\peto x}v(y)
  \le\liminf_{\Om\ni y\peto x}u(y)\not=-\infty\text{ for each } x\in\bdy_P\Om,
\]
then $v\leq u$ in $\Om$.
\end{prop}

\begin{proof}
Let $\eps>0$. Since $u$ is lower semicontinuous and $v$ is upper semicontinuous, we know that
$V_\eps:=\{y\in\Om\, :\, u(y)-v(y)>-\eps\}$ is an open subset of $\Om$.

By the assumption of this proposition, for each $x\in\bdy_P\Om$ 
we can find a neighborhood $V_\eps^x$ of $x$ in 
$\OmPC$ such that $v<u+\eps$ in $V_\eps^x\cap\Om$. Note that $V_\eps^x\cap\Om\subset V_\eps$
for each $x\in\bdy_P\Om$. Thus $V_\eps\cup\bdy_P\Om$ is an open subset of $\overline{\Om}^P$.

Let $U_\eps=\OmPC\sm\overline{V_\eps}$ and 
$C_\eps$ be a connected component of $U_\eps$. Then, by Lemma~\ref{lem:nbhd-bdy}, 
$\overline{C_\eps}^{P,\Om}\subset\Om$ and $v\le u+\eps$ on $\bdy_P^\Om C_\eps$.

By Corollary~\ref{cor:main}, we know that $\overline{C_\eps}\subset\Om$, and, since 
$\bdy_P^\Om C_\eps=\bdy C_\eps$ in this case, $v\le u+\eps$ on $\bdy C_\eps$. 
We now proceed 
as in~\cite[Theorem~7.2]{KM} to see that $v\leq u$ in $C_\eps$. Since this inequality
holds for each connected component of $U_\eps$, we conclude that $v\leq u+\eps$ in $U_\eps$. 
Letting $\eps\to0$, the proof is complete.
\end{proof}

An immediate consequence of Proposition \ref{prop:comp} is the following Corollary.

\begin{cor}\label{cor:upper-lower}
If $f:\bdy_P\Om\to\R$, then
\[
\itunderline{P}_{\overline{\Om}^P}f\leq\itoverline{P}_{\overline{\Om}^P}f.
\]
\end{cor}

Before we state our main theorem, we first need the following two results.
The first part of Proposition~\ref{prop:obstseq} is proved in~\cite{Sh-conv}, while 
the second part is proved in~\cite[Proposition~5.5]{BBSPer}. See~\cite{Far-ObsConv} for more
on convergence properties related to obstacle problems. The proof of Lemma~\ref{lem:decsets} 
is very similar to the proof of the analogous result~\cite[Lemma~7.5]{BBS},
and so we omit the proof here.

\begin{lem}\label{lem:decsets}
Let $\{U_k\}_{k=1}^\infty$ be a decreasing sequence of relatively open sets in $\OmPC$ such that 
$\CpP(U_k)<2^{-kp}$. Then there exists a decreasing sequence of nonnegative functions 
$\{\psi_j\}_{j=1}^\infty$ on $\Om$ such that $||\psi_j||_{\Np(\Om)}<2^{-j}$, $\psi_j\geq k-j$ in 
$U_k\cap\Om$, and $\ds\lim_{\Om\ni y\peto x}\psi_j(y)\geq k-j$ for all $x\in U_k\cap\bdy_P\Om$.
\end{lem}


\begin{prop}\label{prop:obstseq}
Let $\{f_j\}_{j=1}^\infty$ be a $p$-quasieverywhere 
decreasing sequence of functions in $\Np(\Om)$ such that $f_j\to f$ 
in $\Np(\Om)$ as $j\to\infty$. Then $Hf_j$ decreases to $Hf$ locally uniformly in $\Om$.

If $u$ and $u_j$ are solutions to the $\mathcal{K}_{f,f}$ and $\mathcal{K}_{f_j,f_j}$-obstacle problems, 
then $\{u_j\}_{j=1}^\infty$ decreases q.e. in $\Om$ to $u$.
\end{prop}

We now state the main theorem of this paper.

\begin{thm}\label{thm:main2}
Let $f:\OmPC\to\overline{R}$ be a $\CpP$-quasicontinuous function such that $f|_{\Om}$ is in $\Np(\Om)$. 
Then $f$ is resolutive and $P_{\overline{\Om}^P}f=Hf$.
\end{thm}

Having overcome the drawback from the fact that $\partial_P\Om$ may not be compact with the help of 
Proposition~\ref{prop:comp}, the proof of the above main theorem is very similar to that of~\cite[Theorem~7.4]{BBS}.
However, one difference still remains-namely the topology of  $\overline{\Om}^P$ near the boundary $\partial_P\Om$,
which is not as simple as that of the Mazurkiewicz boundary. Hence we provide most of the details of the proof of
Theorem~\ref{thm:main2} here.

\begin{proof}
First, we assume that $f\geq 0$. We extend $Hf$ to $\OmPC$ by letting $Hf=f$ on $\bdy_P\Om$. We now show that this extension is $\CpP$-quasicontinuous.

Let $h=f-Hf$. Then $h\in\Np_0(\Om)$ is quasicontinuous on $\Om$ with $\CpP$-quasicontinuous extension 
$h=0$ to $\bdy_P\Om$, see Proposition~\ref{prop:cpqc}. Because $f$ is $\CpP$-quasicontinuous on 
$\overline{\Om}^P$, it now follows that
so is $Hf$. 

Pick open sets $\{G_j\}$ in $\OmPC$ with $\CpP(G_j)<2^{-jp}$ such that $Hf|_{\OmPC\sm G_j}$ is continuous. 
Defining $U_k=\ds\bigcup_{j=k+1}^\infty G_j$, we see that $\CpP(U_k)<2^{-kp}$ and 
$Hf|_{\OmPC\sm U_k}$ is still continuous.

These sets $\{U_k\}$ fulfill the conditions of Lemma \ref{lem:decsets}, and so we may take 
functions $\{\psi_j\}$ as described in that Lemma. We set $f_j=Hf+\psi_j$ (note 
here that $f_j$ is a function on $\Om$ alone) and let $\phi_j$ be the lower semicontinuously 
regularized solution of the $\mathcal{K}_{f_j,f_j}(\Om)$-obstacle problem as given
in Definition~\ref{def-obst}.

For each positive integer $m$ we have that
\[
f_j\geq\psi_j\geq m\text{ on $U_{m+j}\cap\Om$}.
\]

Given $\eps>0$, let $x\in\bdy_P\Om$. If $x\not\in U_{m+j}$, by the continuity of 
$Hf|_{\OmPC\sm U_{m+j}}$, there is a neighborhood $V_x$ of $x$ in $\OmPC$ such that
\[
f_j(y)\geq Hf(y)\geq Hf(x)-\eps=f(x)-\eps\text{ for all $y\in(V_x\cap\Om)\sm U_{m+j}$}.
\]
So, if $x\in\bdy_P\Om\sm U_{m+j}$,
\[
f_j\geq\min\{f(x)-\eps,m\}\text{ in }V_x\cap\Om^P.
\]
If, instead, $x\in U_{m+j}$, we take $V_x= U_{m+j}$. 

Now, as in the proof of~\cite[Theorem~7.4]{BBS} we have that
$\phi_j(y)\geq\min\{f(x)-\eps,m\}\text{ for all }y\in V_x\cap\Om$.
Therefore,
\[
\liminf_{\Om\ni y\peto x} \phi_j(y)\geq\min\{f(x)-\eps,m\}.
\]
As $\eps\to0$ and $m\to\infty$, we have that
\[
\liminf_{\Om\ni y\peto x}\phi_j(y)\geq f(x)\text{ for all }x\in\bdy_p\Om.
\]
Since $\phi_j$ is \p-superharmonic, we have that $\phi_j\in\mathcal{U}_f(\overline{\Om}^P)$, and so 
$\phi_j\geq\itoverline{P}_{\overline{\Om}^P}f$. 
Because $Hf$ is the solution to the $\mathcal{K}_{Hf,Hf}(\Om)$-obstacle problem, by 
Proposition~\ref{prop:obstseq} we know that $\phi_j$ decreases quasieverywhere to $Hf$, that is, 
$\itoverline{P}_{\overline{\Om}^P}f\leq Hf$ q.e. in $\Om$ when $f\ge 0$.

Note that if $f\in N^{1,p}(\Om)$ has a $\CpP$-quasicontinuous extension to $\overline{\Om}^P$, then
so does $\max\{f,m\}$ for each integer $m$.
Therefore, for $f\in\Np(\Om)$, not necessarily non-negative,
\[
\itoverline{P}_{\overline{\Om}^P}f\leq\lim_{m\to-\infty}\itoverline{P}_{\overline{\Om}^P}\max\{f,m\}
  \le\lim_{m\to\infty}H\max\{f,m\}=Hf\ \text{ q.e.~in } \Om.
\]
Because $\itoverline{P}_{\overline{\Om}^P}f$ is $p$-harmonic in $\Om$ (the proof of this fact,
for the Dirichlet problem corresponding to the metric boundary, can be found in~\cite[Section~4]{BBSPer};
the proof there goes through in our setting verbatim since the modifications considered there in the proof 
occur only on relatively compact subsets of $\Om$ itself)
and hence is continuous, we have that both 
$\itoverline{P}_{\overline{\Om}^P}f$ and $Hf$ are continuous. Therefore 
$\itoverline{P}_{\overline{\Om}^P}f\leq Hf$ everywhere in $\Om$.

Finally, with the aid of Proposition~\ref{prop:comp}, or more precisely, with the help of 
Corollary~\ref{cor:upper-lower}, we see that
\[
\itunderline{P}_{\overline{\Om}^P}f=-\itoverline{P}_{\overline{\Om}^P}(-f)\geq-H(-f)
=Hf\geq\itoverline{P}_{\overline{\Om}^P}f\geq\itunderline{P}_{\overline{\Om}^P}f.
\]
Thus $Hf=\itunderline{P}_{\overline{\Om}^P}f=\itoverline{P}_{\overline{\Om}^P}f$ and $f$ is resolutive.
\end{proof}

The following results show that solution $P_{\overline{\Om}^P}f$ is stable under perturbation of
$f$ on a set of $\CpP$ capacity zero. 
Some of these results have analogous results for the Mazurkiewicz boundary in~\cite{BBS}, but the conclusions found
there are stronger in general, due to the fact that the hypothesis that the Mazurkiewicz boundary (which is the same
as $\partial_{SP}\Om$) is compact is much stronger than the assumptions in our paper.
For example, in the setting of~\cite{BBS} \emph{all} continuous functions are resolutive, whereas
in our setting these continuous functions are in addition required to be Lipschitz on $\partial_{SP}\Om$.
The results here are not duplicates of those in~\cite{BBS} because the
domains considered here do not in general have $\partial_{SP}\Om$ compact, and so the results found in
our paper are valid for a larger class of domains than the results of~\cite{BBS}, including the setting of~\cite{BBS}.

\begin{prop}\label{prop:NpQcont}
Let $f:\OmPC\to\overline{\R}$ be a $\CpP$-quasicontinuous function with $f|_{\Om}$ in the class $\Np(\Om)$. 
If $h:\bdy_P\Om\to\overline{\R}$ is zero $\CpP$ quasi-everywhere, then $f+h$ is resolutive with respect to 
$\overline{\Om}^P$, and $P_{\overline{\Om}^P}(f+h)=P_{\overline{\Om}^P}(f)$.
\end{prop}

\begin{proof}
We may extend $h$ into $\Om$ by zero, and clearly $h|_{\Om}\in\Np(\Om)$. Note that, since 
$\CpP$ is an outer capacity (see Lemma \ref{prop:cppout}), this extended function $h$ is 
$\CpP$-quasicontinuous. 
Thus $f+h$ is $\CpP$-quasicontinuous. Finally, $(f+h)|_{\Om}\in\Np(\Om)$, so by using 
Theorem~\ref{thm:main2}, $f+h$ is resolutive and $P_{\overline{\Om}^P}(f+h)=H(f+h)$. Since 
$f=f+h$ in $\Om$, we therefore have $Hf=H(f+h)$. Thus, by Theorem~\ref{thm:main2} again,
\[
P_{\overline{\Om}^P}(f+h)=H(f+h)=Hf=P_{\overline{\Om}^P}f.
\]
\end{proof}

\begin{cor}\label{cor-ModZeroCap}
Let $f:\OmPC\to\overline{\R}$ be a bounded $\CpP$-quasicontinuous function with 
$f|_{\Om}\in\Np(\Om)$ and $u$ be a bounded \p-harmonic function in $\Om$. 
If $E\subset\bdy_P\Om$ such that $\CpP(E)=0$ and, for all $x\in\bdy_P\Om\sm E$,
\[
\lim_{\Om\ni y\peto x}u(y)=f(x),
\]
then $u=P_{\overline{\Om}^P}f$.
\end{cor}
\begin{proof}
Since both $f$ and $u$ are bounded, we may (simultaneously) rescale them such that 
$0\leq f,u\leq 1$. Then we know that $u\in\mathcal{U}_{f-\chi_E}(\overline{\Om}^P)$ 
and $u\in\mathcal{L}_{f+\chi_E}(\overline{\Om}^P)$. Thus, by the preceding proposition,
\[
u\leq\itunderline{P}_{\overline{\Om}^P}(f+\chi_E)=P_{\overline{\Om}^P}f
  =\itoverline{P}_{\overline{\Om}^P}(f-\chi_E)\leq u.
\]
\end{proof}


Finally, as an application of the above resolutivity results, 
we discuss issues of resolutivity of continuous functions on $\bdy_P\Om$. Note that by
the results in~\cite{BBSPer}, continous functions on $\bdy \Om$ are resolutive. However, in the setting
of $\bdy_P\Om$ we are unable to get such a general result. However, we are able to get resolutivity for
certain types of continuous functions on $\bdy_P\Om$.
This is the focus of the the remaining part of this paper.

Recall that $\bdy_{SP}\Om$ denotes the collection of all prime ends whose impression contains only one point.
As discussed in Section~2 (see also~\cite{ABBS}), this set is equipped with a metric $d_M$, the extension of
the Mazurkiewicz metric on $\Om$.

\begin{prop}\label{ctsres}
Let $f:\bdy_P\Om\to\R$ be continuous on $\bdy_P\Om$ and $d_M$-Lipschitz continuous on $\bdy_{SP}\Om$. 
Then $f$ is resolutive. Furthermore, if $h:\bdy_P\Om\to\overline{\R}$ is zero $\CpP$ quasi-everywhere, 
then $f+h$ is resolutive with respect to  $\overline{\Om}^P$, and $P_{\overline{\Om}^P}(f+h)=P_{\overline{\Om}^P}(f)$.
\end{prop}

\begin{proof}
By an application of the McShane extension theorem (see~\cite{Hei}), we  
extend $f$ to a function $F:\OmPC\to\R$ such that $F=f$ on $\bdy_P\Om$ and $F$ is 
$d_M$-Lipschitz on $\Om\cup\bdy_{SP}\Om$. 

We now show that $F$ is continuous on $\OmPC$. By construction, $F\vert_{\Om\cup\bdy_{SP}\Om}$ is continuous. 
Since $F=f$ on $\bdy_P\Om$, we also see that $F\vert_{\bdy_P\Om}$ is also continuous.
It remains to show that given any end $[\{E_k\}_k]\in\bdy_P\Om\sm\bdy_{SP}\Om$ and a
sequence $\{x_n\}_n$
in $\Om$ with $x_n\peto[\{E_k\}_k]$, we have $F(x_n)\to F([\{E_k\}_k])$. 

At first, we will prove our result only for sequences $x_n\peto[\{E_k\}_k]$ such that, for each $n$, 
$x_n\in N(I[\{E_k\}_k],\frac{1}{n})$. In addition, we will fix a representative chain 
$\{E_k\}_k\in[\{E_k\}_k]$ such that, for all $n\geq k$, $x_n\in E_k$. 
Recall also that we assume $X$ to be a geodesic space.

By modifying the proof of Theorem~\ref{spdense}, we obtain a sequence $\{[\{F_k^n\}_k]\}_n$ in 
$\bdy_{SP}\Om$ such that $[\{F_k^n\}_k]\peto[\{E_k\}_k]$ and $d_M(x_n,[\{F_k^n\}_k])\leq\frac{1}{n}$.

Since $F$ is continuous on $\bdy_P\Om$, we know that $F([\{F_k^n\}_k])\to F([\{E_k\}_k])$. Given 
any $\eps$, we may pick a large-enough positive integer $N$ such that 
\[
|F([\{F_k^N\}_k])-F([\{E_k\}_k])|<\frac{\eps}{2}
\]and
\[
|F(x_N)-F([\{F_k^N\}_k])|\leq Ld_M(x_N,[\{F_k^N\}_k])\leq \frac{L}{N}\leq\frac{\eps}{2},
\] where $L$ is the $d_M$-Lipschitz constant for $F$ on $\Om\cup\bdy_{SP}\Om$. Then
\[
|F(x_N)-F([\{E_k\}_k])|\leq |F([\{F_k^N\}_k])-F([\{E_k\}_k])|+|F(x_N)-F([\{F_k^N\}_k])|\leq\eps.
\] Thus, $F(x_n)\to F([\{E_k\}_k])$.

Now, given any arbitrary sequence $\{x_n\}$ of points in $\Om$ such that $x_n\peto[\{E_k\}_k]$, 
consider $\{|F(x_n)-F([\{E_k\}_k])|\}_n$. Given any subsequence of $\{x_n\}$, we may pick a further 
subsequence $\{z_n\}$ such that $z_n\in N(I[\{E_k\}_k],\frac{1}{n})$. Therefore, 
$|F(z_n)-F([\{E_k\}_k])|\to0$, implying that $|F(x_n)-F([\{E_k\}_k])|\to0$, which completes the 
proof of continuity of $F$.

Now an application of the main theorem above yields the resolutivity of $F$, and hence
the resolutivity of $f$, completing the proof of the first part of the proposition.

The second part now follows from an application of Proposition~\ref{prop:NpQcont} to the function
$F$.
\end{proof}

\begin{rem}\label{rem:rem7.11}
Observe that in the above proposition, we can relax the condition of $f$ being continuous on
$\bdy_P\Om$ to $f$ being $\CpP$-quasicontinuous on $\bdy_P\Om$, the remaining (Lipschitz)
condition of  $f$ also holding. More precisely, if for each $\eps>0$ we can find an open
set $U_\eps\subset\overline{\Omega}^P$ with $\CpP(U_\eps)<\eps$ such that 
$f\vert_{[\bdy_P\Omega\setminus(U_\eps)]\cup\bdy_{SP}\Om}$ is continuous and
$f$ is $d_M$-Lipschitz continuous on $\bdy_{SP}\Om$, then $f$ is resolutive.
\end{rem}

\section{Some examples}
 
  The use of prime ends in the Perron method also yields new results about Euclidean domains.
For example, in the classical Dirichlet problem where the boundary considered is the topological
(that is, metric) boundary of the domain, it is well-known that if $f:\bdy\Om\to\R$ is continuous and
$E\subset\bdy\Om$ such that $C_p(E)=0$, then any bounded perturbation of $f$ on $E$ would yield
a resolutive function whose Perron solution agrees with the Perron solution of $f$ in $\Om$; see~\cite{HKM}
for a proof of this in the weighted Euclidean setting, and for more general metric measure spaces as considered
in this paper, see~\cite{BBSPer} for this fact. The
prime end boundary approach studied in this paper gives a larger set $E$ on which the value of the
boundary data would be irrelevant in the above sense of perturbation. 

The goal of this section is to give three such example domains in Euclidean setting. 


\begin{egs}
The first example we discuss in this section is that of the harmonic comb, also known as the 
topologist's comb. This example was extensively studied in~\cite{B}. This comb is a simply connected
planar domain given by
\[
  \Om:= (0,1)\times(0,1)\, \sm \bigcup_{n\in\mathbb{N}}\{1/n\}\times[0,1/2].
\]
It was shown in~\cite{B} that given a function on $\bdy \Om$,
continuous and bounded on $\bdy\Om\sm \{0\}\times[0,1/2)$, any perturbation of the function
on the set $E:=\{0\}\times[0,1/2)$ yields a resolutive function whose Perron solution coincides with the
Perron solution of the original function. Note that $C_p(E)>0$ for $p>1$, but $\CpP(P(E))=0$.
Note also that the prime end boundary in this case is the same as the singleton prime end boundary
$\partial_{SP}\Omega$. Hence the ``prime end-Perron solution" of any boundary data defined on
$\bdy\Om$ is independent of the values of the boundary data on $E$ as long as the boundary data
is Lipschitz (with respect to the Mazurkiewicz metric $d_M$) continuous on the part of the boundary 
of $\Om$ that arises as impressions of prime ends. On the other hand, if $f$ is a quasicontinuous 
function on $\overline{\Om}\sm \{0\}\times[0,1/2)$ (not necessarily bounded) such that 
$f\vert_\Om\in N^{1,p}(\Om)$, then $f\vert_{\partial\Om\sm\{0\}\times[0,1/2]}$ is resolutive, and any
perturbation of $f$ on a set $F\subset\partial\Om$ with $\CpP(P(F))=0$ yields the same Perron
solution. Hence the results obtained from the perspective of prime end boundaries are complementary
to the results in~\cite{B}. 
\end{egs}

In the above example none of the points in $E$ belongs to the impression of any prime end.
However, 
 $\CpP(P(E))$ does make sense. We point out here that the results of~\cite{B} are related to another
 type of Perron solution, namely, the Perron solution with respect to the topological boundary $\bdy\Om$.
 The difference between the two types of Perron solutions in this instance is that in the case of the 
 prime end boundary, the condition on the superharmonic functions is not enforced at any of the points in
 $E$. This is more in line with the behavior of functions in $N^{1,p}$-classes; boundary points that are not
 accessible via rectifiable curves from within the domain ought not to influence the Dirichlet problem for the 
 domain.

As a consequence of the results of the previous section (see Remark~\ref{rem:rem7.11}), 
if we know that $\CpP(\bdy_P\Om\setminus\bdy_{SP}\Om)=0$,
then any bounded function on $\bdy_P\Om$ that is Lipschitz continuous on $\bdy_{SP}\Om$ with respect to 
the Mazurkiewicz metric $d_M$ is resolutive, and any bounded perturbation of such a function on
$\bdy_P\Om\setminus\bdy_{SP}\Om$ yields a resolutive function whose Perron solution agrees with the 
Perron solution of the original function. This phenomenon is illustrated by the following two examples.

\begin{egs}
This example considers the so-called double comb:
\[
  \Om:= (0,1)\times(0,1)\sm\bigcup_{1<n\in\mathbb{N}}\{1/(2n)\}\times[0,1-1/n]\cup\{1/(2n+1)\}\times[1/n,1].
\]
This again is a simply connected planar domain, but now the set 
$E:=\{0\}\times[0,1]$ is the impression of a single prime end. Note that $\bdy_P\Om$ is compact in this example,
but $\bdy_{SP}\Om$ is not. It is again easy to see (by the use of functions $u_\eps:=\eps \dinn^\Om(x_0,\cdot)$
for each $\eps>0$ and for a fixed $x_0\in\Om$) that $\CpP(\bdy_P\Om\setminus\bdy_{SP}\Om)=0$,
although $C_p(P^{-1}(\bdy_P\Om\setminus\bdy_{SP}\Om))>0$. It follows that
any function on $\bdy_P\Om$ that is Lipschitz continuous on $\bdy_{SP}\Om$ (with respect to $d_M$) is resolutive, and 
any perturbation of this function on $E$ is also resolutive. Strictly speaking, each individual point in
$E$ does not form a separate prime end boundary; the entire set $E$ is the impression of a prime end.
Therefore, in the above statement, by ``perturbation on $E$" we mean perturbing the value of the function
by by changing its value to a different one on the entire set $E$. However, we can relax this "constant on $E$"
requirement in the following sense. Any function on
the topological boundary $\bdy\Om$ that is Lipschitz continuous on $\bdy\Om\setminus E$ is resolutive, and 
perturbations of such functions on $E$ would yield the same Perron solution. A similar statement holds for
functions on $\overline{\Om}$ that are quasicontinuous on $\overline{\Om}\sm E$ such that the restriction of the 
function to $\Om$ belongs to $N^{1,p}(\Om)$.
Such resolutivity result does not follow from the now-classical results in~\cite{BBSPer}. 
\end{egs}

In the above example we had only one element of the prime end boundary that did not belong to
$\bdy_{SP}\Om$. We now construct an example where the set $\bdy_P\Om\setminus\bdy_{SP}\Om$
is uncountable and satisfies $\CpP(\bdy_P\Om\setminus\bdy_{SP}\Om)=0$ while
$C_p(P^{-1}(\bdy_P\Om\setminus\bdy_{SP}\Om))>0$.

\begin{egs}
In this example we consider a domain in $\R^3$:
{\small
\[
   \Om:=(0,1)^3\sm \bigcup_{1<n\in\mathbb{N}}\{1/(2n)\}
      \times[0,3/4+1/n]\times[0,1-1/n]\cup\{1/(2n+1)\}\times[1/4-1/n,1]\times[1/n,1].
\]}
Clearly none of the points in $E:=\{0\}\times[0,1]^2$ is accessible from $\Om$, and it can be shown
using the same technique as in the previous example that
$\CpP(\bdy_P\Om\setminus\bdy_{SP}\Om)=0$, while $C_p(E)>0$. In this case, note that, for each
line segment $\gamma$ in the $2$-dimensional hyperplane region $E$ that connects the line
$\{0\}\times\{1/4\}\times[0,1]$ to the line $\{0\}\times\{3/4\}\times[0,1]$ and lies in between them,
there is a prime end in $\bdy_P\Om$
with that line as its impression. 
Such a prime end is obtained by considering acceptable sets $E_k=\bigcup_{x\in\gamma}B(x,1/k)\cap\Om$.
By the construction of $\Om$, it follows that $E_k$ is connected for each positive integer $k$.
It follows that 
$\bdy_P\Om\setminus\bdy_{SP}\Om$ is uncountable. 
\end{egs}

\begin{rem}
We conclude this section by posing the following two open problems: 
\begin{enumerate}
\item Are there bounded domains that
fail the assumption given in Definition~\ref{St-Assume}?
\item Are there bounded domains for which $\CpP(\bdy_P\Om\setminus\bdy_{SP}\Om)>0$?
\end{enumerate}
\end{rem}


\noindent Address:\\

\noindent D.E.: Department of Mathematical Sciences, P.O.Box 210025, University of
Cincinnati, Cincinnati, OH 45221--0025, U.S.A. \\
\noindent E-mail: {\tt estepdy@mail.uc.edu}, {\tt dewey.estep@uc.edu}
\\

\noindent N.S.: Department of Mathematical Sciences, P.O.Box 210025, University of
Cincinnati, Cincinnati, OH 45221--0025, U.S.A. \\
\noindent E-mail: {\tt shanmun@uc.edu} 

\end{document}